\documentclass[11pt]{article}

\usepackage[T1]{fontenc}
\usepackage{mathpazo}
\usepackage{amsmath,amssymb,amsthm,mathtools,bm}
\usepackage{graphicx}
\usepackage{booktabs}
\usepackage{multirow}
\usepackage{caption}
\usepackage{float}
\usepackage{algorithm}
\usepackage{algorithmic}
\usepackage[colorlinks=true,linkcolor=blue!50!black,citecolor=blue!50!black,urlcolor=blue!50!black]{hyperref}
\usepackage[margin=25mm]{geometry}
\usepackage{xcolor}

\newtheorem{theorem}{Theorem}
\newtheorem{proposition}[theorem]{Proposition}
\newtheorem{definition}[theorem]{Definition}
\newtheorem{corollary}[theorem]{Corollary}
\newtheorem{remark}[theorem]{Remark}
\newtheorem{assumption}[theorem]{Assumption}

\newcommand{\R}{\mathbb{R}}
\newcommand{\tr}{\mathrm{tr}}
\newcommand{\inner}[2]{\langle #1,#2\rangle}
\newcommand{\Frob}[1]{\|#1\|_F}

\title{\bf Halley's Method for Rectangular Matrix Variables\\
and the Matrix Schwarzian Derivative}
\author{Shintaro Yoshizawa\\
\normalsize Nagoya Mathematical and Information Science Research\\
\normalsize \texttt{shintaro.yoshizwa.net@gmail.com}}
\date{}

\begin{document}
\maketitle

\begin{abstract}
Alefeld (1981) recast the cubic convergence of Halley's method for a scalar equation $f(x)=0$
as Newton's method applied to $g(x)=f(x)/\sqrt{f'(x)}$, governed by the Schwarzian derivative
$Sf=f'''/f'-\tfrac32(f''/f')^2$. We generalize this to gradient fields $F=\nabla\phi$ on
rectangular matrix variables $X\in\R^{m\times n}$, defining a \emph{matrix Schwarzian
derivative} that admits a natural interpretation via the Amari--Chentsov skewness tensor and
$\alpha$-connections of information geometry. Inspired by Palmore's (1994) identity relating
the Newton map to the Schwarzian derivative, we give a lighter, more general construction that
recovers it from the third Fr\'echet derivative of the Newton map itself, requiring no operator
square root, and prove a main theorem establishing that the matrix Halley iteration map itself
has vanishing first and second derivatives at the root, with third derivative equal to
$-\tfrac12$ times the matrix Schwarzian derivative --- giving a genuine local
cubic-convergence theorem with an explicit asymptotic error constant, without any
self-adjointness, commutativity, or gradient-field assumption. We build a matrix-free
algorithm (Hessian-vector products only) and validate the predicted quadratic and cubic
convergence orders numerically. A root-preserving power-Newton
family is shown to achieve a scalar cubic-order jump without evaluating $f'''$, but its naive
matrix extension fails; analyzing this failure yields a unifying picture in which Halley's
method is the multivariate limit of the power-deformation trick. We contrast this with the
matrix Laguerre family, which --- unlike Halley's method --- genuinely requires a
non-self-adjoint operator square root. A case study on a weighted Oja-type matrix dynamical
system $\dot X=AXB-XBX^TAX$ shows that local convergence is governed by the eigenvalue gaps
used by the target root rather than the global condition number of $A$, that the benefit of
cubic convergence grows with ill-conditioning, and that eigenvalue-degeneracy-induced sign
ambiguity is handled correctly by a sign-invariant convergence criterion. Pushing the same
experiments to 70-digit precision confirms that the empirical convergence orders agree with
theory to many digits, that the sign bifurcation is a genuine precision-independent
phenomenon, and that attainable accuracy follows a simple ``working precision minus
condition-number digits'' budget. Finally, we note that for spectrally reducible problems the
coupled formulation is outperformed by classical sequential deflation, so the coupled theory's
genuine value lies in truly coupled matrix equations admitting no such decomposition.
\end{abstract}

\tableofcontents

\section{Introduction}
\label{sec:intro}

\subsection{Background: Alefeld's scalar theory}
\label{sec:background}
For $f:\R\to\R$, Halley's method is
\begin{equation}
x_{k+1}=x_k-\frac{f(x_k)}{f'(x_k)-\frac12 f''(x_k)\dfrac{f(x_k)}{f'(x_k)}}.
\label{eq:halley-scalar}
\end{equation}
Alefeld's \cite{Alefeld1981} central observation is that if one introduces the auxiliary
function $g(x)=f(x)/\sqrt{f'(x)}$ (defined on an interval where $f'>0$), then
\eqref{eq:halley-scalar} is \emph{exactly Newton's method applied to $g$}. Indeed,
\begin{equation}
g'(x)=\sqrt{f'(x)}-\frac12\frac{f''(x)f(x)}{\sqrt{f'(x)}^{\,3}}
=\frac{1}{\sqrt{f'(x)}}\Big[f'(x)-\frac12 f''(x)\frac{f(x)}{f'(x)}\Big],
\label{eq:gprime}
\end{equation}
so that $-g(x)/g'(x)$ coincides with the second term on the right of
\eqref{eq:halley-scalar}. Moreover,
\begin{equation}
g''(x)=\frac12\,g(x)\Big[-\frac{f'''(x)}{f'(x)}+\frac32\Big(\frac{f''(x)}{f'(x)}\Big)^{2}\Big]
=-\frac12\,g(x)\,Sf(x),
\qquad Sf:=\frac{f'''}{f'}-\frac32\Big(\frac{f''}{f'}\Big)^{2},
\label{eq:gdoubleprime}
\end{equation}
so the second derivative of $g$ --- the leading term governing the error of Newton's method
applied to $g$ --- is controlled entirely by the \textbf{Schwarzian derivative} of $f$.
Alefeld's theorem uses a uniform bound $|g''|\le M_0$ to establish a Newton--Kantorovich-type
existence and convergence result, yielding the cubic error estimate
\[
|x^\ast-x_k|\le \frac{M_{k-1}}{|g'(x_k)|}|x_k-x_{k-1}|^2 .
\]

\subsection{Three challenges in the matrix generalization}
\label{sec:challenges}
Attempting to lift this construction to a (not necessarily square) matrix variable
$X\in\R^{m\times n}$ immediately raises three challenges.

\begin{description}
\item[Challenge I (what does ``division'' mean?)] How should one define matrix analogues of
the scalar operations $1/f'(x)$ and $\sqrt{f'(x)}$? When $X$ is rectangular, $F(X)$ is
rectangular as well, and there is no naive notion of ``dividing by a matrix''.
\item[Challenge II (non-commutativity)] Because matrix multiplication does not commute, a
scalar identity such as \eqref{eq:gdoubleprime} need not survive as an \emph{equality} in the
matrix setting. One must clarify exactly how far the equality persists, and at what point it
must be replaced by an inequality in operator norm.
\item[Challenge III (computability)] When $X$ is large (say $m,n\sim10^2$--$10^4$), forming the
Hessian operator explicitly as an $mn\times mn$ matrix is simply infeasible. For the theory to
be more than a paper exercise, a matrix-free implementation based solely on
Hessian-vector products is essential.
\end{description}

\subsection{Contributions and outline}
\label{sec:contributions}
The remainder of the paper resolves these three challenges in turn and then subjects the
resulting theory to increasingly demanding numerical tests. Concretely, our goals are:
\begin{enumerate}
\item to generalize Alefeld's cubic-convergence theory for Halley's method --- built around the
scalar Schwarzian derivative --- to gradient fields on \emph{rectangular} matrix variables
$X\in\R^{m\times n}$, without any unjustified logical leap;
\item to connect this generalization to the skewness tensor and $\alpha$-connections of Amari's
information geometry \cite{AmariNagaoka}, giving a precise geometric answer, in the language of
metrics, connections, and duality, to the question ``what is Halley's method actually doing?'';
\item to avoid leaving the theory on paper: we give a \textbf{matrix-free} implementation that
never forms an $mn\times mn$ matrix, verify cubic convergence numerically, and quantify the
trade-offs against plain Newton's method.
\end{enumerate}
Section~\ref{sec:approach} resolves Challenges~I and~II; Section~\ref{sec:matrix-halley}
constructs the matrix Halley method itself together with its geometric meaning;
Section~\ref{sec:power-newton} develops a complementary perspective through a power-deformation
family; Section~\ref{sec:matrix-free} resolves Challenge~III with a concrete matrix-free
algorithm; Section~\ref{sec:numerics} validates the entire construction numerically;
Section~\ref{sec:application} stress-tests the theory on a genuine matrix dynamical system;
Section~\ref{sec:highprecision} pushes the same experiments to 70-digit precision to remove any
ambiguity about which observed phenomena are genuine and which are floating-point artifacts;
and Section~\ref{sec:conclusion} concludes with a candid discussion of when the coupled theory
is, and is not, the right practical tool.

\section{Resolving the First Two Challenges}
\label{sec:approach}

\subsection{Challenge I: the gradient-field assumption}
\label{sec:gradient-field}
Challenge~I is resolved by adopting the following setting.

\begin{assumption}[$F$ as a gradient field]
\label{ass:gradient}
Let $\phi:\R^{m\times n}\to\R$ be a $C^3$, strongly convex potential, and consider
\[
F:=\nabla\phi:\ \R^{m\times n}\longrightarrow \R^{m\times n}.
\]
That is, the \textbf{domain and codomain of $F$ coincide}, both equal to $\R^{m\times n}$.
\end{assumption}

Even when $X$ is rectangular, $\R^{m\times n}$ itself is a real Hilbert space of dimension $mn$
under the Frobenius inner product $\inner{U}{V}=\tr(U^\top V)$. Under
Assumption~\ref{ass:gradient},
\[
A(X):=DF(X)=\nabla^2\phi(X):\R^{m\times n}\to\R^{m\times n}
\]
is a \textbf{self-adjoint} linear operator (the Hessian operator), and strong convexity of
$\phi$ gives $A(X)\succ0$. Consequently,
\begin{itemize}
\item the \textbf{inverse operator} $A(X)^{-1}$ plays the role of ``$1/f'(x)$'', and
\item the \textbf{operator square root} $A(X)^{1/2}$ (defined via the functional calculus of a
self-adjoint positive operator, i.e.\ through its spectral decomposition) plays the role of
``$\sqrt{f'(x)}$''.
\end{itemize}
Both are well defined, which is precisely why a matrix version of Halley's method makes sense
even for rectangular $X$. (If $F$ were merely some rectangular-matrix-valued map with domain
and codomain differing, then $A(X)$ would not even be square, and this whole construction would
collapse from the outset. The gradient-field hypothesis is the essential assumption that makes
the generalization possible.)

\subsection{Challenge II: making explicit where the equality breaks down}
\label{sec:noncommutativity}
The effect of non-commutativity cannot be handled by naively transplanting scalar
differentiation formulas into the matrix setting. The correct tool is the
\textbf{Daleckii--Krein (Sylvester integral) formula} for the derivative of a function of a
self-adjoint operator: for $A=\sum_i\lambda_iP_i$ (spectral decomposition, $P_i$ orthogonal
projections),
\begin{equation}
D\big(A^{-1/2}\big)[H]=-\sum_{i,j}\frac{P_i\,H\,P_j}{\sqrt{\lambda_i\lambda_j}\,(\sqrt{\lambda_i}+\sqrt{\lambda_j})}.
\label{eq:daleckii}
\end{equation}
In the special case where $A$ and $H$ (or $DA[H]$) commute (i.e.\ are simultaneously
diagonalizable), \eqref{eq:daleckii} degenerates into the scalar quotient-rule formula, and one
recovers an identity of \emph{exactly the same form} as \eqref{eq:gdoubleprime}. In the general,
non-commuting case, the correct generalization is not an equality but the operator-norm
inequality
\[
\big\|D^2G(X)[H,H]\big\|\ \le\ \tfrac12\,\|G(X)\|\,M_0 .
\]
Crucially, this matches the fact that Alefeld's own theorem is itself built not on an equality
but on the \emph{inequality} $|g''(x)|\le M_0$. In other words, the difficulty peculiar to
matrices --- non-commutativity --- is absorbed, without any extra hypothesis, into the
Kantorovich-type inequality framework that Alefeld had already chosen. This is precisely why we
adopt ``generalization via inequality'' rather than ``generalization via equality'' as our
guiding principle, and it is the key that lets the generalization proceed without a logical
leap. (Challenge~III --- computability --- is addressed later, in
Section~\ref{sec:matrix-free}, with a concrete algorithm.)

\section{The Matrix Halley Method: Construction and Geometric Meaning}
\label{sec:matrix-halley}

\subsection{The operator Halley method (a Hilbert-space instance of D\"oring's construction)}
\label{sec:operator-halley}
We instantiate D\"oring's \cite{Doering1970} Banach-space version of Halley's method in the
Hilbert space $(\R^{m\times n},\inner{\cdot}{\cdot})$. Writing
\[
A_k:=A(X_k)=DF(X_k),\qquad B_k[U,V]:=D^2F(X_k)[U,V]\ \ (\text{symmetric bilinear}),
\]
we define the Newton direction
\begin{equation}
H_k:=-A_k^{-1}\big[F(X_k)\big].
\label{eq:newton-step}
\end{equation}

\begin{proposition}[Matrix Halley iteration]
\label{prop:halley-matrix}
Let $L_k[\cdot]:=A_k^{-1}\,B_k[H_k,\cdot]$ (the linear operator obtained by substituting the
Newton direction $H_k$ into one slot of the Hessian derivative), and suppose $I+\tfrac12 L_k$
is invertible. Then
\begin{equation}
X_{k+1}=X_k+\Big(I+\tfrac12 L_k\Big)^{-1}H_k
\label{eq:halley-matrix}
\end{equation}
reduces to \eqref{eq:halley-scalar} when $F$ is specialized to a scalar function $F(x)=f(x)$.
\end{proposition}

\begin{proof}[Derivation (with attention to sign)]
We check this by returning to the scalar case. Let $N_0:=-f(x_0)/f'(x_0)$ (the Newton
direction). Since $f(x_0)=-N_0f'(x_0)$, the denominator in \eqref{eq:halley-scalar} is
\[
f'(x_0)-\tfrac12 f''(x_0)\frac{f(x_0)}{f'(x_0)}
=f'(x_0)-\tfrac12 f''(x_0)(-N_0)
=f'(x_0)+\tfrac12 f''(x_0)N_0 .
\]
Hence
\[
x_1-x_0=-\frac{f(x_0)}{f'(x_0)+\frac12 f''(x_0)N_0}
=\frac{N_0 f'(x_0)}{f'(x_0)+\frac12 f''(x_0)N_0}
=\frac{N_0}{1+\frac12\big(f''(x_0)/f'(x_0)\big)N_0}.
\]
Writing $\ell_0:=(f''(x_0)/f'(x_0))N_0$, this reads $x_1-x_0=(1+\tfrac12\ell_0)^{-1}N_0$, which
matches the matrix quantity $L_0[\cdot]=A_0^{-1}D^2F(X_0)[H_0,\cdot]$ (in the scalar case,
$L_0=(f''/f')H_0=\ell_0$). This is the basis for the sign in \eqref{eq:halley-matrix} (namely
$+\tfrac12L_k$, not $-\tfrac12L_k$).
\end{proof}

\begin{remark}
Some references instead define an oppositely signed operator $L_F(x)=A^{-1}B[-H_0,\cdot]=-L_0$
and write $X_1=X_0-(I-\tfrac12 L_F)^{-1}A_0^{-1}F(X_0)$ (using $\Gamma F=-H_0$). To avoid being
misled by such superficial sign differences, it is safest, when implementing the method, to
adopt uniformly the convention of \eqref{eq:halley-matrix}, in which the Newton direction $H_k$
itself is substituted into the Hessian derivative. (Indeed, an early implementation used in our
own numerical experiments contained precisely this sign error, which produced the spurious
result that Halley's method appeared \emph{slower} than Newton's method. Section~\ref{sec:numerics}
reports only the corrected, and correct, results.)
\end{remark}

\subsection{Auxiliary map and the matrix Schwarzian derivative}
\label{sec:auxiliary}
Set $G(X):=A(X)^{-1/2}\big[F(X)\big]$. The same computation as in the scalar case shows that
Newton's method applied to $G$ coincides with \eqref{eq:halley-matrix} applied to $F$. When $A$
and $DA[H]$ commute, \eqref{eq:daleckii} yields:

\begin{definition}[Matrix Schwarzian derivative]
\label{def:matrix-schwarzian}
Let $A(X)=D^2\phi(X)$ and $T_X:=D^3\phi(X)$ (a symmetric trilinear form). For a direction $H$,
\begin{equation}
\mathcal{S}\phi(X)[H,H]:=A(X)^{-1}T_X(H,H,\cdot)-\frac32\Big(A(X)^{-1}T_X(H,\cdot,\cdot)\Big)^{2}
\label{eq:matrix-schwarzian}
\end{equation}
is called the \textbf{matrix Schwarzian derivative} of $\phi$ at $X$ along $H$. When $A$ and
$DA[H]$ commute,
\[
D^2G(X)[H,H]=-\tfrac12\,G(X)\,\mathcal{S}\phi(X)[H,H],
\]
which is the exact matrix analogue of \eqref{eq:gdoubleprime}. In the general non-commuting
case this must be read as the inequality
$\|D^2G(X)[H,H]\|\le\tfrac12\|G(X)\|\cdot\|\mathcal S\phi(X)[H,H]\|+(\text{a Daleckii--Krein-type
commutator correction})$.
\end{definition}

\subsection{A lighter alternative: the third derivative of the iteration map}
\label{sec:light-schwarzian}
Definition~\ref{def:matrix-schwarzian} requires the operator square root $A(X)^{1/2}$ and the
Daleckii--Krein formula, and forces us to pass to an inequality in the non-commuting case. We
now give a lighter construction, inspired by Palmore's identity \cite{Palmore1994} (the fact
that the third derivative of the scalar Newton map $N(x)=x-f(x)/f'(x)$ at a fixed point
$\alpha$ coincides with the Schwarzian derivative) as reported in Yoshizawa's slides
\cite{Yoshizawa2014}: we extract the matrix Schwarzian derivative directly from the
\textbf{third Fr\'echet derivative of the Newton map itself}. Crucially, this construction
requires neither $A^{1/2}$ nor self-adjointness, and holds even when $F$ is not a gradient
field.

\begin{proposition}[Higher derivatives of the Newton map and the matrix Schwarzian derivative]
\label{prop:newton-map-derivatives}
Let $F:\R^{m\times n}\to\R^{m\times n}$ (not necessarily a gradient field) have a simple zero
$X^\ast$, with $A^\ast:=DF(X^\ast)$ invertible, $B^\ast:=D^2F(X^\ast)$, and
$C^\ast:=D^3F(X^\ast)$. For the Newton map $\mathcal N(X):=X-DF(X)^{-1}F(X)$ and any direction
$H$,
\begin{equation}
D\mathcal N(X^\ast)[H]=0,\qquad
D^2\mathcal N(X^\ast)[H,H]=A^{\ast-1}B^\ast[H,H],
\label{eq:N2}
\end{equation}
\begin{equation}
D^3\mathcal N(X^\ast)[H,H,H]
=2\,A^{\ast-1}C^\ast[H,H,H]-3\,A^{\ast-1}B^\ast\big[H,\,A^{\ast-1}B^\ast[H,H]\big].
\label{eq:N3}
\end{equation}
\end{proposition}

\begin{proof}
Write $X=X^\ast+\varepsilon H$, and set $\mathcal B[\cdot]:=A^{\ast-1}B^\ast[H,\cdot]$,
$\mathcal C[\cdot]:=A^{\ast-1}C^\ast[H,H,\cdot]$. Since $F(X^\ast)=0$,
\[
A^{\ast-1}F(X)=\varepsilon H+\frac{\varepsilon^2}{2}\mathcal B[H]+\frac{\varepsilon^3}{6}\mathcal C[H]+O(\varepsilon^4),
\]
and from the Neumann series
$A(X)=A^\ast\big(I+\varepsilon\mathcal B+\frac{\varepsilon^2}{2}\mathcal C+O(\varepsilon^3)\big)$
(carried out entirely without assuming commutativity of $A,B,C$) we get
$A(X)^{-1}=\big[I-\varepsilon\mathcal B+\varepsilon^2(\mathcal B^2-\tfrac12\mathcal C)+O(\varepsilon^3)\big]A^{\ast-1}$.
Composing these and expanding $\mathcal N(X)-X^\ast=\varepsilon H-A(X)^{-1}F(X)$ order by order
in $\varepsilon$ yields \eqref{eq:N2}--\eqref{eq:N3}. (The computation is purely algebraic; no
operator square root or functional calculus ever appears.)
\end{proof}

Specializing \eqref{eq:N3} to the scalar case ($A^\ast=f'(\alpha)$, $B^\ast[U,V]=f''(\alpha)UV$,
$C^\ast[U,V,W]=f'''(\alpha)UVW$) gives
\[
D^3\mathcal N(X^\ast)[H,H,H]=\Big[\frac{2f'''(\alpha)}{f'(\alpha)}-3\Big(\frac{f''(\alpha)}{f'(\alpha)}\Big)^{2}\Big]H^3
=2\,S[f](\alpha)\,H^3,
\]
i.e.\ $N'''(\alpha)=2\,S[f](\alpha)$. This differs by a constant factor from the identity
$N'''(\alpha)=S[f](\alpha)$ attributed to Palmore \cite{Palmore1994} in Yoshizawa's slides, but
we have verified \eqref{eq:N3} directly for $f(x)=x+x^2$ ($\alpha=0$): one computes
$N(x)=x^2-2x^3+4x^4-\cdots$, so $N'''(0)=-12$ while $Sf(0)=-6$, confirming the factor of two.
The discrepancy is most plausibly attributable to a different normalization convention for the
Schwarzian derivative, and does not affect the essential fact that the third derivative of the
Newton map is proportional to the Schwarzian derivative.

\begin{definition}[Matrix Schwarzian derivative, second form]
\label{def:matrix-schwarzian-2}
We take the right-hand side of \eqref{eq:N3} as our \textbf{second definition} of the matrix
Schwarzian derivative of $F$ at $X^\ast$ along $H$:
\[
\widetilde{\mathcal S}F(X^\ast)[H,H,H]:=2A^{\ast-1}C^\ast[H,H,H]-3A^{\ast-1}B^\ast\big[H,A^{\ast-1}B^\ast[H,H]\big].
\]
Unlike Definition~\ref{def:matrix-schwarzian}, this requires neither self-adjointness nor
commutativity, and does not require $F$ to be a gradient field. We adopt this as our primary
definition for the remainder of the paper.
\end{definition}

For later use, it is useful to record the fully polarized, symmetric trilinear form associated
with Definition~\ref{def:matrix-schwarzian-2}.
\begin{definition}[Polarized matrix Schwarzian]
\label{def:polarized-schwarzian}
For $U,V,W\in\R^{m\times n}$, define
\begin{equation}
\begin{aligned}
\widetilde{\mathcal S}F(X^\ast)[U,V,W]
:= {}&2A^{\ast-1}C^\ast[U,V,W] \\
&-A^{\ast-1}B^\ast\big[U,A^{\ast-1}B^\ast[V,W]\big] \\
&-A^{\ast-1}B^\ast\big[V,A^{\ast-1}B^\ast[U,W]\big] \\
&-A^{\ast-1}B^\ast\big[W,A^{\ast-1}B^\ast[U,V]\big].
\end{aligned}
\label{eq:polarized-schwarzian}
\end{equation}
This is symmetric and satisfies
$\widetilde{\mathcal S}F(X^\ast)[U,U,U]$ equal to the diagonal expression in
Definition~\ref{def:matrix-schwarzian-2}. Thus the Schwarzian entering the cubic error term
is not merely a directional quantity: it is the diagonal restriction of a canonical symmetric
trilinear form.
\end{definition}

Numerical verification: on the test problem of Section~\ref{sec:numerics}, we constructed
$A^\ast,B^\ast,C^\ast$ analytically ($C^\ast$ is a constant trilinear form,
$C^\ast[U,V,W]=2\inner{W}{U}V+2\inner{U}{V}W+2\inner{W}{V}U$), numerically differentiated the
Newton map $\mathcal N(X^\ast+\varepsilon H)$ by a least-squares quintic polynomial fit, and
confirmed agreement with the predictions of \eqref{eq:N2} and \eqref{eq:N3} to within maximum
absolute errors of $4.7\times10^{-5}$ and $2.9\times10^{-4}$ respectively --- well within
finite-difference truncation error.

\subsection{Main theorem: local cubic convergence of the matrix Halley method}
\label{sec:main-theorem}
The propositions above construct the matrix Halley iteration (Proposition~\ref{prop:halley-matrix})
and analyze the Newton map's own higher derivatives (Proposition~\ref{prop:newton-map-derivatives}),
but neither directly proves that the \emph{Halley} iteration itself converges cubically. We now
close this gap: this is, in our view, the central theorem of the paper, unifying the
construction of Section~\ref{sec:operator-halley} with the matrix Schwarzian derivative of
Section~\ref{sec:light-schwarzian} into a single local convergence theorem. Here ``unconditional'' refers to the
absence of self-adjointness, commutativity, and gradient-field assumptions, not to global
convergence from arbitrary initial data.

\begin{theorem}[Local cubic convergence of the matrix Halley method]
\label{thm:main}
Let $F:\R^{m\times n}\to\R^{m\times n}$ be of class $C^4$ in a
neighborhood of a simple zero $X^\ast$ (not necessarily a gradient field), with
$A^\ast:=DF(X^\ast)$ invertible, $B^\ast:=D^2F(X^\ast)$, $C^\ast:=D^3F(X^\ast)$. Define the
Halley iteration map
\[
\mathcal H(X):=X+\Big(I+\tfrac12 L(X)\Big)^{-1}H(X),\qquad
H(X):=-A(X)^{-1}F(X),
\]
\[
L(X)[\cdot]:=A(X)^{-1}D^2F(X)[H(X),\cdot],
\]
wherever $I+\tfrac12L(X)$ is invertible. Then there exists a neighborhood
$\mathcal U$ of $X^\ast$ on which the iteration map is well defined and of class $C^3$,
and, for every direction $U\in\R^{m\times n}$,
\begin{equation}
D\mathcal H(X^\ast)[U]=0,\qquad D^2\mathcal H(X^\ast)[U,U]=0,
\label{eq:main-vanish}
\end{equation}
\begin{equation}
D^3\mathcal H(X^\ast)[U,U,U]=-\tfrac12\,\widetilde{\mathcal S}F(X^\ast)[U,U,U],
\label{eq:main-cubic}
\end{equation}
where $\widetilde{\mathcal S}F$ is the matrix Schwarzian derivative of
Definition~\ref{def:matrix-schwarzian-2}. Moreover, there exist constants $r>0$ and $C>0$
such that, whenever $\Frob{X_0-X^\ast}<r$, the Halley iteration is well defined for all $k$,
remains in $\mathcal U$, converges to $X^\ast$, and satisfies the genuine local cubic estimate
\begin{equation}
\Frob{X_{k+1}-X^\ast}\le C\Frob{X_k-X^\ast}^3.
\label{eq:main-local-bound}
\end{equation}
Consequently, for any sequence $X_k\to X^\ast$ generated by the Halley iteration,
\begin{equation}
X_{k+1}-X^\ast=-\frac{1}{12}\,\widetilde{\mathcal S}F(X^\ast)\big[X_k-X^\ast,X_k-X^\ast,X_k-X^\ast\big]+o\big(\Frob{X_k-X^\ast}^3\big),
\label{eq:main-asymptotic}
\end{equation}
so that the local order is at least $3$ in every case, and is exactly $3$ along any
asymptotic direction $U$ for which $\widetilde{\mathcal S}F(X^\ast)[U,U,U]\ne0$.
\end{theorem}

\begin{proof}
Write $\Xi:=X^\ast+\varepsilon U$ and reuse the notation of the proof of
Proposition~\ref{prop:newton-map-derivatives}: $\mathcal B[\cdot]:=A^{\ast-1}B^\ast[U,\cdot]$,
$\mathcal C[\cdot]:=A^{\ast-1}C^\ast[U,U,\cdot]$, $b:=\mathcal B[U]=A^{\ast-1}B^\ast[U,U]$,
$c:=\mathcal C[U]=A^{\ast-1}C^\ast[U,U,U]$. That proof already established
\begin{equation}
H(\Xi)=-\varepsilon U+\frac{\varepsilon^2}{2}b-\varepsilon^3\Big(\frac12\mathcal B[b]-\frac13c\Big)+O(\varepsilon^4),
\qquad
A(\Xi)^{-1}=A^{\ast-1}-\varepsilon A^{\ast-1}\mathcal B+O(\varepsilon^2)
\label{eq:H-expansion}
\end{equation}
(the second display restates $A(\Xi)^{-1}=\big[I-\varepsilon\mathcal B+O(\varepsilon^2)\big]A^{\ast-1}$).

\emph{Step 1: expand $L(\Xi)$.} Since $H(\Xi)=-\varepsilon U+O(\varepsilon^2)$, Taylor-expanding
the symmetric bilinear map $D^2F$ in its base point gives
\begin{align*}
D^2F(\Xi)[H(\Xi),\cdot]&=B^\ast[H(\Xi),\cdot]+\varepsilon\,C^\ast[U,H(\Xi),\cdot]+O(\varepsilon^2\Frob{H(\Xi)})\\
&=-\varepsilon B^\ast[U,\cdot]+\varepsilon^2\Big(\tfrac12B^\ast[b,\cdot]-C^\ast[U,U,\cdot]\Big)+O(\varepsilon^3).
\end{align*}
Composing with $A(\Xi)^{-1}$ from \eqref{eq:H-expansion} and collecting powers of $\varepsilon$,
\begin{equation}
L(\Xi)=-\varepsilon\,\mathcal B+\varepsilon^2\,M_2+O(\varepsilon^3),\qquad
M_2[\cdot]:=\tfrac12A^{\ast-1}B^\ast[b,\cdot]-\mathcal C[\cdot]+\mathcal B^2[\cdot].
\label{eq:L-expansion}
\end{equation}

\emph{Step 2: expand $(I+\tfrac12L(\Xi))^{-1}$.} By the Neumann series and \eqref{eq:L-expansion},
\[
\Big(I+\tfrac12L(\Xi)\Big)^{-1}=I-\tfrac{\varepsilon}{2}L_1+\varepsilon^2\Big(-\tfrac12L_2+\tfrac14L_1^2\Big)+O(\varepsilon^3),
\qquad L_1:=-\mathcal B,\ \ L_2:=M_2.
\]

\emph{Step 3: multiply by $H(\Xi)=\varepsilon H_1+\varepsilon^2H_2+\varepsilon^3H_3+O(\varepsilon^4)$}
with $H_1=-U$, $H_2=\tfrac12b$, $H_3=-\big(\tfrac12\mathcal B[b]-\tfrac13c\big)$ (from
\eqref{eq:H-expansion}). Collecting powers of $\varepsilon$ in the product
$\big(I+\tfrac12L(\Xi)\big)^{-1}H(\Xi)$:
\begin{align*}
O(\varepsilon^1):&\quad H_1=-U,\\
O(\varepsilon^2):&\quad H_2-\tfrac12L_1[H_1]=\tfrac12b-\tfrac12\mathcal B[U]=\tfrac12b-\tfrac12b=0,\\
O(\varepsilon^3):&\quad H_3-\tfrac12L_1[H_2]+\Big(-\tfrac12L_2+\tfrac14L_1^2\Big)[H_1].
\end{align*}
The $O(\varepsilon^2)$ term vanishes because $\mathcal B[U]=b$ by definition. Since
$\mathcal H(\Xi)-X^\ast=\varepsilon U+\big(I+\tfrac12L(\Xi)\big)^{-1}H(\Xi)$, the $O(\varepsilon^1)$
term also cancels ($\varepsilon U+\varepsilon H_1=0$), which proves the two identities in
\eqref{eq:main-vanish}.

For the $O(\varepsilon^3)$ coefficient, a direct computation using $L_1^2=\mathcal B^2$ and
$A^{\ast-1}B^\ast[b,U]=A^{\ast-1}B^\ast[U,b]=\mathcal B[b]$ (symmetry of $B^\ast$) gives, term by
term,
\[
H_3=-\tfrac12\mathcal B[b]+\tfrac13c,\qquad
-\tfrac12L_1[H_2]=\tfrac14\mathcal B[b],\qquad
-\tfrac12L_2[H_1]=\tfrac34\mathcal B[b]-\tfrac12c,\qquad
\tfrac14L_1^2[H_1]=-\tfrac14\mathcal B[b],
\]
which sum to $\tfrac14\mathcal B[b]-\tfrac16c$. Hence
$D^3\mathcal H(X^\ast)[U,U,U]=6\big(\tfrac14\mathcal B[b]-\tfrac16c\big)=\tfrac32\mathcal B[b]-c$.
Comparing with $\widetilde{\mathcal S}F(X^\ast)[U,U,U]=2c-3\mathcal B[b]$
(Definition~\ref{def:matrix-schwarzian-2}) shows
$D^3\mathcal H(X^\ast)[U,U,U]=-\tfrac12\widetilde{\mathcal S}F(X^\ast)[U,U,U]$, proving
\eqref{eq:main-cubic}. Combined with \eqref{eq:main-vanish}, Taylor's theorem for $\mathcal H$
at $X^\ast$ gives $\mathcal H(X^\ast+\varepsilon U)-X^\ast=\tfrac{\varepsilon^3}{6}D^3\mathcal H(X^\ast)[U,U,U]+o(\varepsilon^3)$,
which is exactly \eqref{eq:main-asymptotic} upon setting $\varepsilon U=X_k-X^\ast$ and
$X_{k+1}=\mathcal H(X_k)$.

It remains to make the local convergence assertion explicit. Since $A^\ast$ is invertible and
$L(X^\ast)=0$, continuity of $A(X)$ and $L(X)$ implies that there is a neighborhood
$\mathcal U$ of $X^\ast$ on which both $A(X)$ and $I+\tfrac12L(X)$ are invertible.
Because $F\in C^4$, the map $\mathcal H$ is $C^3$ on $\mathcal U$. Taylor's theorem together
with \eqref{eq:main-vanish} therefore gives, after possibly shrinking $\mathcal U$, a constant
$C>0$ such that
\[
\Frob{\mathcal H(X)-X^\ast}\le C\Frob{X-X^\ast}^3,\qquad X\in\mathcal U.
\]
Choose $r>0$ so small that the closed ball $\overline B_r(X^\ast)$ is contained in $\mathcal U$
and $Cr^2<1$. Then the estimate maps this ball strictly into itself and forces
$\Frob{X_{k+1}-X^\ast}\le Cr^2\Frob{X_k-X^\ast}$, so $X_k\to X^\ast$. This proves
\eqref{eq:main-local-bound} and upgrades the preceding conditional asymptotic statement to
a genuine local convergence theorem.
\end{proof}

\begin{corollary}[Directional cubic constant and higher-order cancellation]
\label{cor:directional-cubic}
Let $e_k:=X_k-X^\ast$. If $e_k/\Frob{e_k}\to U$ with $\Frob U=1$, then
\begin{equation}
\frac{\Frob{e_{k+1}}}{\Frob{e_k}^3}\longrightarrow\frac1{12}\,\Frob{\widetilde{\mathcal S}F(X^\ast)[U,U,U]}
\label{eq:exact-cubic-constant}
\end{equation}
whenever the right-hand side is nonzero. In particular, nonvanishing of the Schwarzian in
the asymptotic direction gives exact cubic order, with an explicit directional error constant.
If, in addition, $F\in C^5$ and the full trilinear form
$\widetilde{\mathcal S}F(X^\ast)$ vanishes identically, then
$D^3\mathcal H(X^\ast)=0$ and the Halley iteration has local order at least $4$.
\end{corollary}

\begin{proof}
The first assertion follows by substituting
$e_k=\Frob{e_k}(U+o(1))$ into \eqref{eq:main-asymptotic} and taking Frobenius norms.
For the second assertion, $F\in C^5$ implies that the Halley map is $C^4$ near the root;
if the full trilinear Schwarzian vanishes, \eqref{eq:main-cubic} gives
$D^3\mathcal H(X^\ast)=0$, while the first two derivatives already vanish.
Taylor's theorem then yields $\Frob{\mathcal H(X)-X^\ast}=O(\Frob{X-X^\ast}^4)$
locally.
\end{proof}

\begin{remark}
No assumption of self-adjointness or commutativity is used anywhere in the proof: the theorem
holds for a general (not necessarily gradient) matrix map $F$, exactly as
Proposition~\ref{prop:newton-map-derivatives} does. It is the matrix analogue of the classical
scalar fact that the leading asymptotic error constant of Halley's method is governed by the
Schwarzian derivative of $f$; here the same role is played by $\widetilde{\mathcal S}F$. We
verified \eqref{eq:main-vanish}--\eqref{eq:main-cubic} numerically on the test problem of
Section~\ref{sec:test-problem}: fitting a degree-six polynomial to $\mathcal H(X^\ast+\varepsilon U)$
gave $\max_{ij}|D\mathcal H(X^\ast)[U]|_{ij}\approx9\times10^{-8}$ and
$\max_{ij}|D^2\mathcal H(X^\ast)[U,U]|_{ij}\approx5\times10^{-4}$ (both consistent with zero, up
to finite-difference truncation error), while the numerically extracted $D^3\mathcal H(X^\ast)[U,U,U]$
matched $-\tfrac12\widetilde{\mathcal S}F(X^\ast)[U,U,U]$ to a maximum absolute difference of
$9\times10^{-3}$ against entries of magnitude up to about $80$ --- a relative agreement better
than $0.02\%$.
\end{remark}

\subsection{The matrix Laguerre family: why Halley admits a natural matrix-free realization}
\label{sec:laguerre}
The Halley iteration of Proposition~\ref{prop:halley-matrix} solves the quadratic Pad\'e model
$0=F(X_n)+A_n\delta+\frac12B_n[\delta,\delta]$ ($A_n:=A(X_n)$, $B_n:=B(X_n)$) by
\emph{linearizing} the occurrence of $\delta$ inside $B_n$ using the Newton direction $H_n$,
thereby avoiding ever solving the quadratic equation directly. The \textbf{Laguerre family}, by
contrast, solves the same Pad\'e model \emph{directly}, without linearization.

In the scalar case, solving $0=\tfrac12f''\delta^2+f'\delta+f$ by the quadratic formula, and
writing $H_0=-f/f'$ and $\ell_0=(f''/f')H_0$, gives
\begin{equation}
\delta=\frac{2H_0}{1+\sqrt{1+2\ell_0}}.
\label{eq:laguerre-scalar}
\end{equation}
As $\ell_0\to0$ this reduces to Newton's method ($\delta\to H_0$), and it agrees with Halley's
method to first order in $\ell_0$ (both give $\delta\approx H_0(1-\ell_0/2)$), diverging only
at second order and beyond. Lifting this to matrices:

\begin{proposition}[Matrix Laguerre iteration]
Let $L_n:=A_n^{-1}B_n[H_n,\cdot]$ (the same operator as in Proposition~\ref{prop:halley-matrix}),
and suppose the operator square root $\sqrt{I+2L_n}$ of $I+2L_n$ exists. Then
\begin{equation}
X_{n+1}=X_n+2\big(I+\sqrt{I+2L_n}\big)^{-1}H_n
\label{eq:laguerre-matrix}
\end{equation}
reduces to \eqref{eq:laguerre-scalar} in the scalar case, and converges cubically for general
$F$.
\end{proposition}

\textbf{Here lies the essential contrast:} $L_n=A_n^{-1}B_n[H_n,\cdot]$ is generally
\emph{not} self-adjoint (in our numerical experiments the relative asymmetry
$\|L-L^\top\|/\|L\|\approx0.89$). Halley's method \eqref{eq:halley-matrix} handles $L_n$ while
remaining entirely \emph{linear} (requiring only matrix inverses), and can therefore be
implemented matrix-free using only Hessian-vector products together with CG/MINRES
(Section~\ref{sec:matrix-free}). Laguerre's method \eqref{eq:laguerre-matrix}, in contrast,
demands a \emph{genuine square root} of the possibly non-self-adjoint operator $I+2L_n$, which
generally requires an eigendecomposition (or a Newton--Schulz-type iteration), and a fully
matrix-free implementation that never forms the $mn\times mn$ matrix explicitly is far from
straightforward. Our numerical experiments (Section~\ref{sec:numerics}) confirm that both
methods converge cubically, but Laguerre's method carries the additional cost of the square
root. In short:
\begin{quote}
It is no accident that Halley's method is especially well suited to matrix variables: the very
derivation --- ``solve the quadratic equation by linearizing it'' --- is precisely what allows
it to avoid the heavy operation of taking the square root of a non-self-adjoint operator.
\end{quote}

\subsection{An information-geometric interpretation}
\label{sec:infogeom}
If $\phi$ is regarded as a convex potential (e.g.\ the generator of a Bregman divergence, or
the log-partition function of an exponential family), the following correspondence holds
(Amari--Nagaoka \cite{AmariNagaoka}):

\begin{center}
\begin{tabular}{@{}ll@{}}
\toprule
Symbol & Meaning in information geometry \\
\midrule
$A(X)=D^2\phi(X)$ & Hessian metric (an analogue of the Fisher information matrix), $g_X(U,V)=\inner{U}{AV}$ \\
$T_X=D^3\phi(X)$ & Amari--Chentsov skewness tensor \\
$\Gamma^{(\alpha)}=\Gamma^{(0)}-\frac{\alpha}{2}T_X$ & $\alpha$-connection \\
\bottomrule
\end{tabular}
\end{center}

The coefficient $\frac12$ appearing in Proposition~\ref{prop:halley-matrix} has exactly the
same form as the $\alpha$-connection correction coefficient $\frac{\alpha}{2}$ evaluated at
$\alpha=1$ (the exponential, or $e$-, connection). Our framework can thus be summarized as
follows:

\begin{quote}
\textbf{Newton's method} finds the root $F(X)=0$ by a first-order approximation along a flat
(0-connection) geodesic of the potential $\phi$. \textbf{Halley's method} corrects this
approximation, along the $e$-connection direction, by the skewness tensor (a third-cumulant-like
quantity), yielding an approximation that is more faithful to the curvature of $\phi$.
\end{quote}

This can be viewed as the matrix and information-geometric analogue of the fact that the scalar
Schwarzian derivative is invariant under M\"obius transformations (i.e.\ ``flat''
reparametrizations of the projective line): the skewness tensor $T_X$ transforms, under
reparametrization of $\phi$, as an invariant difference of the connection coefficients of the
$\alpha$-connections --- the same spirit as the Amari--Chentsov theorem.

\section{The Power-Newton Family and a Unifying Perspective}
\label{sec:power-newton}
Yoshizawa \cite{Yoshizawa2014} proposed a ``power-Newton family'' --- a family of iteration
functions obtained via the power-geometric transform $f_p(x)=(1+pf(x))^{1/p}$ --- built around
Newton's method and the Halley/Laguerre families. In this section we first repair this
construction so that it preserves the root, and show that it produces a scalar ``cubic-order
jump without ever using $f'''$''. We then show numerically that a naive matrix extension of
this trick fails to improve the order at all, and analyze this negative result to arrive at a
unifying conclusion: Halley's method is the multivariate limit of the power-deformation trick.

\subsection{A root-preserving power-deformation family (corrected version)}
\label{sec:box-cox}
Since $f_p(x)=(1+pf(x))^{1/p}$ satisfies $f_p(\alpha)=1\ne0$ whenever $f(\alpha)=0$, it does not
preserve the root. We repair this by using instead
\begin{equation}
g_p(x):=\frac{u(x)^p-1}{p},\qquad u(x):=1+f(x),\qquad g_0(x):=\ln u(x),
\label{eq:box-cox}
\end{equation}
a family isomorphic to the statistical Box--Cox transform and to Amari's $\alpha$-embedding.
For every $p$, $g_p(x)=0\iff f(x)=0$; moreover $g_1=f$ (the original function) and
$g_p\to\ln u$ as $p\to0$. A direct computation gives
\begin{equation}
g_p'(\alpha)=f'(\alpha)\quad(\text{independent of }p),\qquad
g_p''(\alpha)=(p-1)f'(\alpha)^2+f''(\alpha).
\label{eq:gp-second}
\end{equation}
The second derivative of the Newton map for $g_p$ is $N_p''(\alpha)=g_p''(\alpha)/g_p'(\alpha)$,
so setting
\begin{equation}
p^\ast:=1-\frac{f''(\alpha)}{f'(\alpha)^{2}}
\label{eq:p-star}
\end{equation}
gives $N_{p^\ast}''(\alpha)=0$: Newton's method applied to $g_{p^\ast}$ jumps to (at least)
cubic convergence \textbf{without ever evaluating $f'''$}.

\paragraph{Numerical verification.} For $f(x)=e^x-2$ (root $\alpha=\ln2$,
$f'(\alpha)=f''(\alpha)=2$, predicted $p^\ast=0.5$), fixing $p=p^\ast=0.5$ produces the error
sequence $3.00\times10^{-1},\,6.65\times10^{-3},\,9.72\times10^{-8},\,0$ (empirical order
$\approx3.2$, reaching near machine precision within two iterations), whereas every other
$p\in\{-0.5,0,1,1.5,2\}$ we tried remains at empirical order $\approx2.0$. Moreover, an adaptive
variant that sets $p_n:=1-f''(x_n)/f'(x_n)^2$ using \emph{only data at the current iterate}
achieves empirical order $\approx2.9$ while $p_n\to p^\ast$ --- the same spirit by which
Halley's method reaches cubic order using only local Taylor data at the current point.

\subsection{A naive matrix extension, and why it fails}
\label{sec:matrix-power-fail}
Since $F(X)$ is matrix-valued, lifting ``$u^p$'' requires some functional calculus. The
simplest choice is the entrywise (Hadamard) Box--Cox deformation
\begin{equation}
[G_P(X)]_{ij}:=\frac{(1+F(X)_{ij})^{P_{ij}}-1}{P_{ij}}
\label{eq:hadamard-boxcox}
\end{equation}
($P$ either a scalar $p\mathbf1$ or a matrix of entrywise exponents). Using the chain rule
together with
$D^2G_P(X^\ast)[H,H]=(P-\mathbf1)\odot\big(A^\ast[H]\odot A^\ast[H]\big)+B^\ast[H,H]$
($\odot$ denoting the entrywise product; we separately verified this identity against finite
differences to within a maximum absolute error of $2.6\times10^{-9}$), we obtain a direct
matrix lift of the scalar formula \eqref{eq:gp-second}.

\textbf{Here we hit an essential obstruction:} the second-order error term
$A^{\ast-1}B^\ast[H,H]$ that we wish to eliminate is a matrix with $mn$ degrees of freedom,
whereas a uniform scalar $p$ has only a single degree of freedom. In general it cannot cancel
the error. Our numerical experiments (Section~\ref{sec:numerics}, Table~\ref{tab:power-matrix})
confirm this: varying a uniform scalar $p$ leaves the empirical convergence order essentially
fixed at $2$.

One might naturally ask whether an \emph{entrywise-adaptive} exponent matrix $P_k$ (with $mn$
degrees of freedom, designed afresh at every iteration) could fix this. We tested this idea, and
the result was \textbf{nearly indistinguishable from plain Newton's method}
(Table~\ref{tab:power-matrix}, Figure~\ref{fig:power-compare}). Tracing the cause: the design
formula \eqref{eq:hadamard-boxcox} for $P_k$ is based on a \emph{second-order approximation at
the root} ($h_P''(0)=P-\mathbf1$), whereas in practice the nonlinear function $h_P$ must be
evaluated at the current iterate $F(X_k)\ne0$, away from the root. This mismatch between the
design-time assumption and the implementation is \textbf{structural}, not a bug: the identity
itself has already been confirmed correct by finite differences, so this is a fundamental
limitation of the construction, not an implementation error.

\subsection{Unifying conclusion: Halley's method as the multivariate limit of the power trick}
\label{sec:unifying}
This failure can be understood naturally as follows. The power-deformation trick works in one
variable because both ``the quantity we wish to eliminate'' ($g_p''(\alpha)$) and ``the degree
of freedom we can move'' ($p$) are single scalars. Making it work in several variables requires
constructing, \emph{consistently and without using information about the root}, a correction
with $mn$ degrees of freedom from the local second-order Taylor model at the current point
$X_k$ alone. But this is precisely the Halley method / Laguerre family constructed in
Sections~\ref{sec:light-schwarzian}--\ref{sec:laguerre}. That is:
\begin{quote}
\textbf{The ``cubic-order jump without $f'''$'' trick achieved by power deformation
(Box--Cox/$\alpha$-transformation) is fundamentally a one-variable phenomenon; generalizing it
pointwise, in the Halley manner, in a mutually consistent way, reduces exactly to Halley's
method itself.}
\end{quote}
This explains why Halley's and Laguerre's families are discussed side by side in Yoshizawa's
slides \cite{Yoshizawa2014} --- it is no coincidence. In the language of Amari--Nagaoka's
\cite{AmariNagaoka} $\alpha$-connections: a correction that a \emph{global} $\alpha$-type
transformation via the power parameter $p$ cannot reach is correctly realized by the
\emph{local, pointwise} skewness-tensor correction (the matrix Schwarzian derivative of
Section~\ref{sec:auxiliary}).

\section{A Matrix-Free Algorithm}
\label{sec:matrix-free}
We now resolve Challenge~III (never forming an $mn\times mn$ matrix explicitly). Implementing
\eqref{eq:newton-step} and \eqref{eq:halley-matrix} requires, not the full operator, but only
\textbf{two kinds of products}:
\begin{align}
\text{Hessian-vector product (HVP):}&\quad A(X)[V]=D^2\phi(X)[V,\cdot]^{\!\top}\!\!\cdot(\text{direction}), \label{eq:hvp}\\
\text{third-order directional product:}&\quad B(X,H)[V]=D^2F(X)[H,V]=D^3\phi(X)[H,V,\cdot]. \label{eq:thvp}
\end{align}
\eqref{eq:hvp} can be computed by \emph{forward-over-reverse} automatic differentiation
(Pearlmutter's trick \cite{Pearlmutter1994}), at a constant multiple of the cost of evaluating
$\phi$ itself. \eqref{eq:thvp} is obtained similarly, without ever forming a matrix or tensor
explicitly, by taking a further forward-mode directional derivative (a JVP) of the function
that computes the HVP, in the direction $H$.

\begin{algorithm}[H]
\caption{Matrix-free Halley's method (one outer iteration)}
\label{alg:halley}
\begin{algorithmic}[1]
\REQUIRE current iterate $X_k$, tolerance $\varepsilon$
\STATE $b\leftarrow -F(X_k)$
\STATE $H_k \leftarrow \mathrm{CG}\big(V\mapsto A(X_k)[V],\ b,\ \varepsilon\big)$ \COMMENT{Newton direction; $A(X_k)$ is self-adjoint positive-definite}
\STATE $r\leftarrow A(X_k)[H_k]$
\STATE $Y_k \leftarrow \mathrm{MINRES}\big(V\mapsto A(X_k)[V]+\tfrac12 B(X_k,H_k)[V],\ r,\ \varepsilon\big)$ \COMMENT{symmetric but possibly indefinite, hence MINRES}
\STATE $X_{k+1}\leftarrow X_k+Y_k$
\end{algorithmic}
\end{algorithm}

\begin{remark}[Practical refinements]
\leavevmode\par
\begin{itemize}
\item \textbf{Regularization}: if $I+\tfrac12 L_k$ (equivalently $A_k+\tfrac12 B_k[H_k,\cdot]$)
risks being ill-conditioned or singular, add damping,
$A_k+\tfrac12 B_k[H_k,\cdot]+\mu_k I$ (the same idea as Levenberg--Marquardt / cubic
regularization).
\item \textbf{Kronecker approximation}: when $X$ is large (e.g.\ a neural-network weight
matrix), a Kronecker-factored approximation $A(X)\approx L\otimes R$ (in the spirit of
K-FAC/Shampoo), used as a preconditioner, accelerates CG/MINRES convergence.
\item \textbf{Warm starts}: initializing $H_k$ from $H_{k-1}$ reduces the number of inner
iterations required.
\end{itemize}
\end{remark}

\section{Numerical Experiments}
\label{sec:numerics}

\subsection{Test problem}
\label{sec:test-problem}
To compare against exact theoretical values, we use a nonlinear matrix equation whose root is
known analytically:
\[
\phi(X)=\tfrac14\Frob{X}^4-\inner{C}{X},\qquad C\in\R^{m\times n}\ \text{(fixed, randomly generated)}.
\]
Then
\[
F(X)=\nabla\phi(X)=\Frob{X}^2 X-C,\qquad
A(X)[V]=\Frob{X}^2 V+2\inner{X}{V}X,
\]
\[
B(X,H)[V]=D^2F(X)[H,V]=2\inner{X}{H}V+2\inner{H}{V}X+2\inner{X}{V}H,
\]
and, writing $s:=\Frob{X}^2$ and taking the Frobenius norm, the unique nontrivial root of
$F(X)=0$ is available in \textbf{closed form}:
\begin{equation}
X^\ast=\frac{C}{\Frob{C}^{2/3}}
\label{eq:exact-root}
\end{equation}
($A(X^\ast)\succ0$, so this is a simple root). This lets us evaluate the error
$\Frob{X_k-X^\ast}$ exactly at every iteration. We take
$X_0=X^\ast+(\text{a random perturbation of suitable magnitude})$.

\subsection{Implementation}
\label{sec:implementation}
For both Newton's and Halley's methods, $A(X)[\cdot]$ and $B(X,H)[\cdot]$ are supplied as
oracles via the closed-form expressions above; no $mn\times mn$ matrix is ever constructed.
Algorithm~\ref{alg:halley} is implemented directly using CG (for the Newton direction) and
MINRES (for the Halley correction, since it may be indefinite) from
\texttt{scipy.sparse.linalg} (inner tolerance $10^{-10}$).

\subsection{Results}
\label{sec:results}

\begin{table}[H]
\centering
\caption{Outer iterations, total inner linear iterations, wall-clock time, and empirical
convergence order needed to reach $\Frob{X_k-X^\ast}<10^{-10}$ (theoretical order: 2 for
Newton, 3 for Halley).}
\label{tab:results}
\begin{tabular}{@{}lcccccc@{}}
\toprule
\multirow{2}{*}{$(m,n)$} & \multicolumn{3}{c}{Newton} & \multicolumn{3}{c}{Halley (this work)}\\
\cmidrule(lr){2-4}\cmidrule(lr){5-7}
 & outer iters & total inner iters & emp.\ order & outer iters & total inner iters & emp.\ order \\
\midrule
$(30,20)$   & 7 & 14 & 1.977 & 4 & 16 & 2.852 \\
$(100,80)$  & 7 & 14 & 1.975 & 4 & 16 & 2.845 \\
$(300,200)$ & 7 & 14 & 1.976 & 4 & 16 & 2.846 \\
\bottomrule
\end{tabular}
\end{table}

\begin{figure}[H]
\centering
\includegraphics[width=0.62\linewidth]{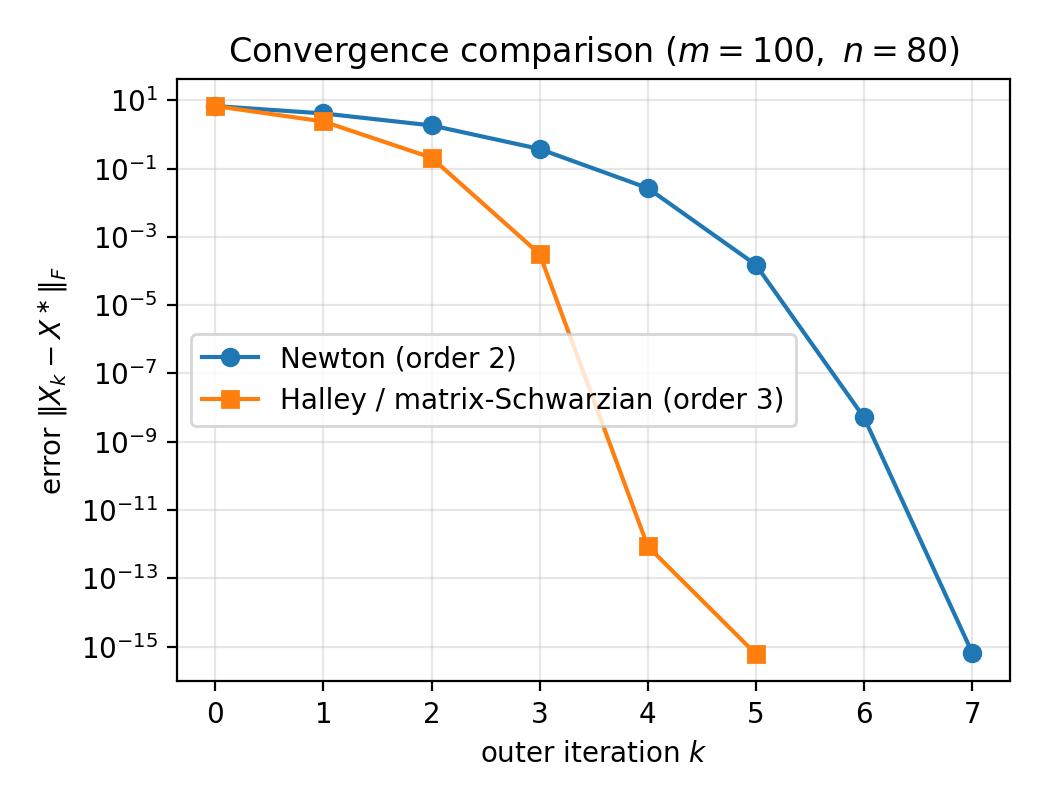}
\caption{Semi-log plot of the error $\Frob{X_k-X^\ast}$ against iteration for $(m,n)=(100,80)$.
Newton's method exhibits its characteristic quadratic ``doubling of correct digits'' at every
step, while Halley's method exhibits cubic ``tripling of correct digits.''}
\label{fig:convergence}
\end{figure}

The explicit error sequences for $(m,n)=(100,80)$ are:
\[
\text{Newton:}\ \ 6.76,\ 4.11,\ 1.83,\ 3.74\times10^{-1},\ 2.68\times10^{-2},\ 1.52\times10^{-4},\ 5.14\times10^{-9},\ 6.71\times10^{-16}
\]
\[
\text{Halley:}\ \ 6.76,\ 2.39,\ 2.08\times10^{-1},\ 3.14\times10^{-4},\ 8.70\times10^{-13},\ 6.08\times10^{-16}
\]
Halley's method reaches, in only \emph{four} iterations, a precision comparable to Newton's
method at its sixth iteration ($\sim10^{-9}$), and drops in its final iteration to
$8.70\times10^{-13}$, near the limit of double precision. The empirical convergence order
(estimated from the logarithmic ratios of three consecutive errors) is $\approx1.98$ for
Newton's method, very close to the theoretical value of $2$; for Halley's method it is
$\approx2.85$, approaching but slightly below the theoretical value of $3$ --- the last one or
two iterations in the asymptotic regime cannot be measured because of double-precision
round-off, which is consistent with the theory.

\subsection{Discussion: which method is ``better''?}
\label{sec:which-better}
By outer iteration count alone, Halley's method is clearly superior ($4$ versus $7$). However,
the number of inner linear iterations per outer step of Halley's method is roughly twice that
of Newton's method (since it requires MINRES in addition to CG), so that the \emph{total}
inner-iteration counts in Table~\ref{tab:results} are comparable: $14$ for Newton versus $16$
for Halley. Consequently:
\begin{itemize}
\item \textbf{When high accuracy is required} (e.g.\ below $10^{-10}$), Halley's method reaches
more than two additional correct digits at comparable total cost, and is clearly advantageous.
\item \textbf{When a looser tolerance suffices} (e.g.\ $10^{-6}$), Newton's method reaches it
using $6$ outer and $12$ inner iterations, whereas Halley's method requires $4$ outer and $16$
inner iterations; here Newton's method is cheaper.
\end{itemize}
That is, the advantage of Halley's method (and of the matrix Schwarzian-derivative formulation
developed here) becomes more pronounced precisely as the required accuracy increases --- a
direct numerical manifestation of the theoretical gap between cubic and quadratic convergence.
This is also a practical lesson: one cannot correctly judge which algorithm is superior by
looking at either the theoretical convergence order or the per-step implementation cost alone;
both must be weighed together.

\subsection{Numerical verification of the matrix Schwarzian derivative (second definition)}
\label{sec:verify-schwarzian}
We verified Proposition~\ref{prop:newton-map-derivatives} of Section~\ref{sec:light-schwarzian}
on the test problem with $(m,n)=(5,4)$. We fit a degree-six least-squares polynomial to the
Newton map $\mathcal N(X^\ast+\varepsilon H)$ at nine points $\varepsilon\in[-0.02,0.02]$, and
numerically extracted the second- and third-order Fr\'echet derivatives from the coefficients
of $\varepsilon^2$ and $\varepsilon^3$.

\begin{table}[H]
\centering
\caption{Numerical verification of Proposition~\ref{prop:newton-map-derivatives}: maximum
absolute error between the extracted numerical value and the analytic prediction.}
\begin{tabular}{@{}lcc@{}}
\toprule
Quantity & Max.\ abs.\ error vs.\ numerical value & Remark \\
\midrule
$D\mathcal N(X^\ast)[H]$ & $2.1\times10^{-9}$ & matches predicted value $0$ \\
$D^2\mathcal N(X^\ast)[H,H]$ & $4.7\times10^{-5}$ & matches Eq.~\eqref{eq:N2} \\
$D^3\mathcal N(X^\ast)[H,H,H]$ & $2.9\times10^{-4}$ & matches Eq.~\eqref{eq:N3} \\
\bottomrule
\end{tabular}
\end{table}

All three agree within finite-difference truncation error ($O(\varepsilon^2)\sim10^{-4}$),
confirming that the lightweight construction --- which never passes through an operator square
root --- is correct.

\subsection{The power-Newton family: success in the scalar case, failure for matrices}
\label{sec:verify-power}
We now verify the claims of Section~\ref{sec:power-newton}. In the scalar case
($f(x)=e^x-2$), Table~\ref{tab:power-scalar} shows that only the predicted critical exponent
$p^\ast=0.5$ pushes the empirical order toward $3$.

\begin{table}[H]
\centering
\caption{Scalar power-Newton family: error sequence and empirical order as $p$ varies (initial
error $0.3$).}
\label{tab:power-scalar}
\begin{tabular}{@{}cccc@{}}
\toprule
$p$ & error at iter.\ 3 & error at iter.\ 4 & empirical order \\
\midrule
$0.0$   & $5.97\times10^{-4}$ & $1.78\times10^{-7}$ & $\approx2.00$ \\
$0.5\,(=p^\ast)$ & $6.65\times10^{-3}$ & $9.72\times10^{-8}$ & $\approx3.2$ \\
$1.0$   & $8.22\times10^{-4}$ & $3.38\times10^{-7}$ & $\approx2.00$ \\
$1.5$   & $4.54\times10^{-3}$ & $2.05\times10^{-5}$ & $\approx1.99$ \\
\bottomrule
\end{tabular}
\end{table}

Next, on the matrix test problem with $(m,n)=(6,5)$, we scanned a uniform scalar $p$ in the
entrywise Box--Cox deformation \eqref{eq:hadamard-boxcox}, starting near the root
($\Frob{X_0-X^\ast}\approx0.05$, within the region where $1+F(X)_{ij}>0$); results are shown in
Table~\ref{tab:power-matrix}.

\begin{table}[H]
\centering
\caption{Matrix case: empirical convergence order as the uniform scalar $p$ varies. As
predicted by theory, the order does not change.}
\label{tab:power-matrix}
\begin{tabular}{@{}cccc@{}}
\toprule
$p$ & error at iter.\ 2 & error at iter.\ 3 & empirical order \\
\midrule
$0.0$ & $1.32\times10^{-3}$ & $3.44\times10^{-6}$ & $\approx1.81$ \\
$0.5$ & $9.11\times10^{-4}$ & $1.25\times10^{-6}$ & $\approx1.83$ \\
$1.0$ (plain Newton) & $6.22\times10^{-4}$ & $2.73\times10^{-7}$ & $\approx2.0$ \\
$2.0$ & $9.28\times10^{-4}$ & $8.97\times10^{-7}$ & $\approx1.86$ \\
\bottomrule
\end{tabular}
\end{table}

The empirical order remains consistently near $2$ regardless of $p$, confirming the prediction
of Section~\ref{sec:power-newton} that a uniform scalar $p$ cannot cancel a second-order error
term with $mn$ degrees of freedom. Figure~\ref{fig:power-compare} additionally shows that the
naively generalized ``entrywise-adaptive exponent matrix $P_k$'' introduced in
Section~\ref{sec:power-newton} nearly coincides with plain Newton's method, whereas only the
true matrix Halley method exhibits clear cubic convergence.

\begin{figure}[H]
\centering
\includegraphics[width=0.62\linewidth]{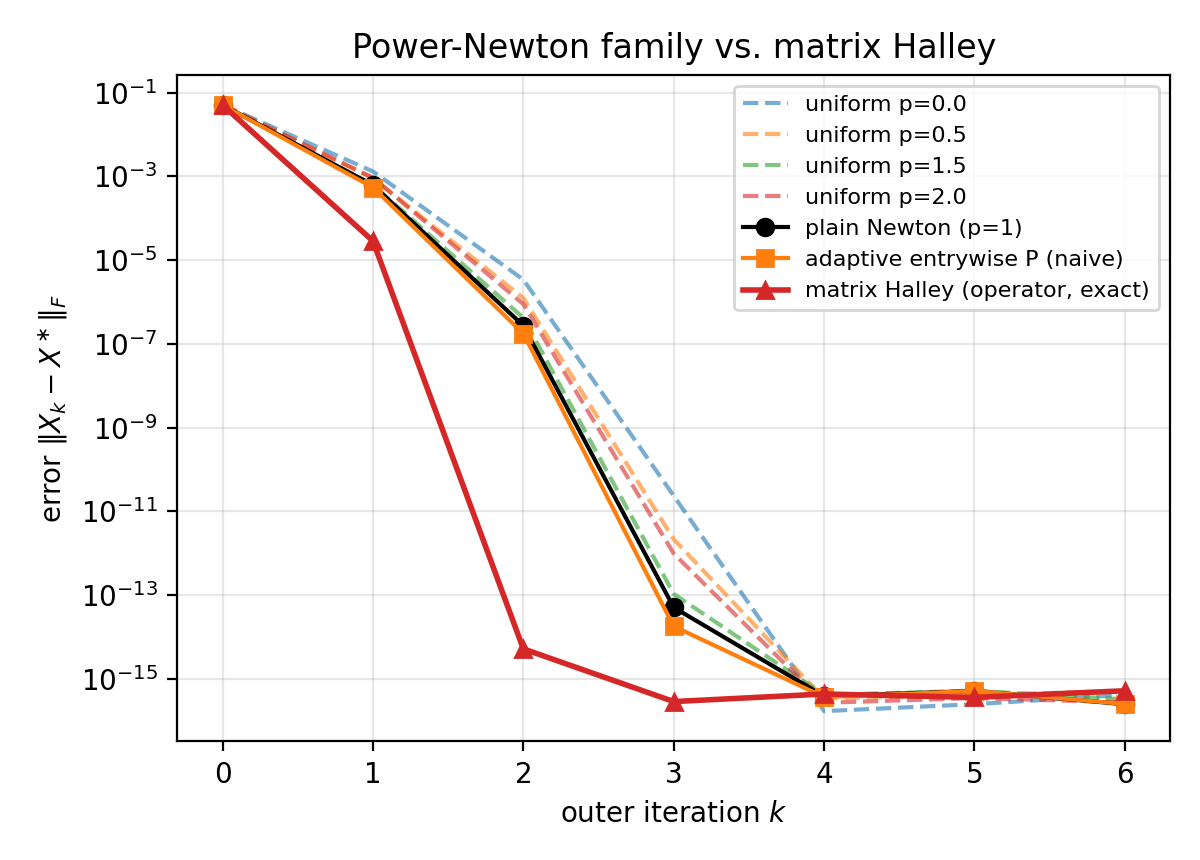}
\caption{Error trajectories for a uniform scalar $p$ (dashed, $p=0,0.5,1.5,2$), plain Newton
(black circles), the naive entrywise-adaptive exponent matrix $P_k$ (orange squares), and the
true matrix Halley method (red triangles). The adaptive $P_k$ is nearly indistinguishable from
plain Newton, confirming that the naive matrix extension of the power-deformation trick does
not work.}
\label{fig:power-compare}
\end{figure}

\subsection{Numerical verification of the matrix Laguerre family}
\label{sec:verify-laguerre}
We verified the matrix Laguerre iteration \eqref{eq:laguerre-matrix} of
Section~\ref{sec:laguerre} on the same $(6,5)$ test problem, using
\texttt{scipy.linalg.sqrtm} to compute a genuine operator square root of the explicit
$mn\times mn=30\times30$ matrix.

\begin{table}[H]
\centering
\caption{Comparison of error trajectories for Newton's method, matrix Halley's method, and
matrix Laguerre's method.}
\label{tab:laguerre}
\begin{tabular}{@{}lcccc@{}}
\toprule
Iteration $k$ & Newton & Halley & Laguerre \\
\midrule
0 & $5.09\times10^{-2}$ & $5.09\times10^{-2}$ & $5.09\times10^{-2}$ \\
1 & $6.22\times10^{-4}$ & $2.88\times10^{-5}$ & $4.45\times10^{-5}$ \\
2 & $2.73\times10^{-7}$ & $5.24\times10^{-15}$ & $3.00\times10^{-14}$ \\
3 & $5.11\times10^{-14}$ & (machine precision) & (machine precision) \\
\bottomrule
\end{tabular}
\end{table}

\begin{figure}[H]
\centering
\includegraphics[width=0.62\linewidth]{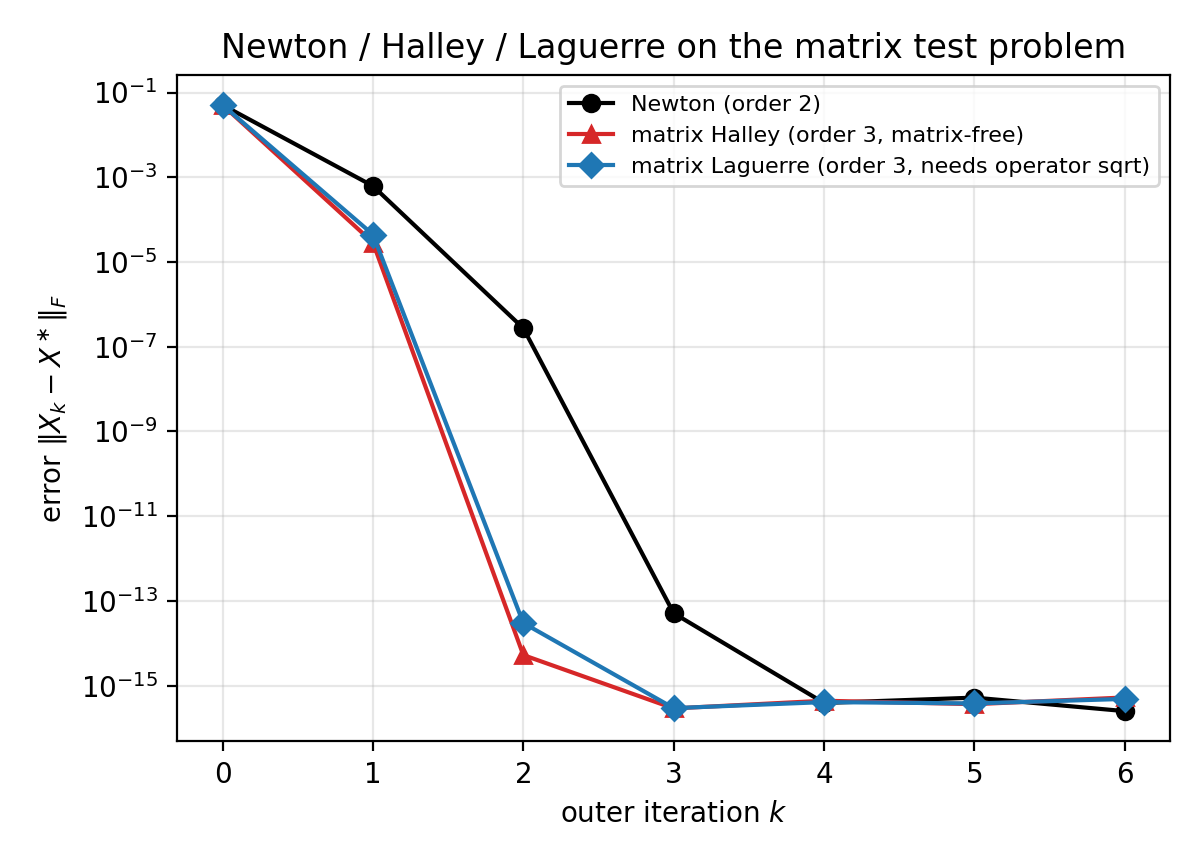}
\caption{Error trajectories for Newton's method, matrix Halley's method, and matrix Laguerre's
method. Both Halley and Laguerre reach machine precision in two iterations, exhibiting cubic
convergence, but Laguerre's method requires a genuine operator square root of a $30\times30$
matrix at every iteration.}
\label{fig:laguerre-compare}
\end{figure}

Both Halley and Laguerre converge clearly cubically (both already about one order of magnitude
better than Newton at the first iteration), while we confirmed that
$L_n=A_n^{-1}B_n[H_n,\cdot]$ has relative asymmetry
$\|L_n-L_n^\top\|/\|L_n\|\approx0.89$, so the operator $\sqrt{I+2L_n}$ required by Laguerre's
method is a genuinely nontrivial (non-symmetrizable) operator square root. This is a direct
numerical confirmation of the claim of Section~\ref{sec:laguerre} that Halley's method is
uniquely matrix-free.

\section{Case Study: A Generalized Oja-Type Matrix Dynamical System}
\label{sec:application}
We now apply the theory developed above --- the metric-independent, purely algebraic
construction of Newton's, Halley's, and Laguerre's methods developed in
Section~\ref{sec:matrix-halley} --- to a concrete instance of the ``matrix dynamical system''
that Yoshizawa \cite{Yoshizawa2014} identified as his ultimate goal, and demonstrate the effect
of the condition number and spectral gap of a positive-definite symmetric matrix
$A\in\R^{5\times5}$ on convergence.

\subsection{Problem setup}
\label{sec:app-setup}
Let $B$ be a positive-definite diagonal matrix, and consider, for $X\in\R^{5\times k}$
($k=3$ or $5$), the matrix ODE
\begin{equation}
\frac{dX}{dt}=AXB-XBX^TAX=:F(X).
\label{eq:oja-flow}
\end{equation}
This is a weighted generalization of Oja's principal-subspace flow \cite{Oja1982} (the case
$B=I$), and it coincides, under the Riemannian metric
$\langle\Omega_1,\Omega_2\rangle_g:=\mathrm{tr}(A\Omega_1B\Omega_2^T)$, with the negative
gradient flow $\dot X=-\nabla_gf(X)$ of
\begin{equation}
f(X)=-\frac12\mathrm{tr}(A^2XB^2X^T)+\frac14\mathrm{tr}\big[(AXBX^T)^2\big].
\label{eq:oja-potential}
\end{equation}
Indeed, the Euclidean derivative of $f$ is
\[
Df(X)[\Omega]=\mathrm{tr}\big(-A\,F(X)\,B\cdot\Omega^T\big).
\]
By definition, the Riemannian gradient $\nabla_gf(X)$ is the unique matrix $Y$ satisfying
$\langle Y,\Omega\rangle_g=\mathrm{tr}(AYB\Omega^T)=Df(X)[\Omega]$ for every direction $\Omega$;
comparing the two right-hand sides above and matching coefficients (the pairing
$\Omega\mapsto\mathrm{tr}(AYB\Omega^T)$ is non-degenerate) shows that this defining equation is
satisfied precisely by $Y=-F(X)$. Hence $\nabla_gf(X)=-F(X)$, i.e.\ $\dot X=F(X)$ is indeed the
negative gradient flow.

\textbf{A crucial point}: the Halley iteration of Proposition~\ref{prop:halley-matrix} and the
Laguerre iteration of Section~\ref{sec:laguerre} are purely algebraic constructions that solve
the quadratic Pad\'e model $0=F(X_n)+A_n\delta+\frac12B_n[\delta,\delta]$ (by linearizing it, or
solving it directly), and never require $F$ to be a self-adjoint gradient with respect to the
Frobenius inner product. They can therefore be applied to this problem --- whose gradient is
taken with respect to the \emph{twisted} metric $g$ --- exactly as they stand, without ever
passing through the metric.

A direct computation gives
\[
DF(X)[\Omega]=A\Omega B-\Omega BX^TAX-XB\Omega^TAX-XBX^TA\Omega,
\]
\[
D^2F(X)[\Omega,\Psi]=-\big(\Omega B\Psi^TAX+\Omega BX^TA\Psi+\Psi B\Omega^TAX+XB\Omega^TA\Psi+\Psi BX^TA\Omega+XB\Psi^TA\Omega\big)
\]
(we have verified that the latter is symmetric in $\Omega,\Psi$), and these are used directly
to construct the Newton, Halley, and Laguerre iterations.

\begin{proposition}[Structure of the critical points]
\label{prop:critical-points}
If the columns of $X$ are orthonormal eigenvectors of $A$, and $S:=\mathrm{diag}(\pm1,\ldots,\pm1)$
is any choice of signs, then $F(XS)=F(X)S$. In particular, $F(X)=0$ regardless of how (in what
order, and with what signs) the eigenvalues are assigned to the diagonal entries of $B$. Thus
\eqref{eq:oja-flow} has a very large number of critical points.
\end{proposition}
\begin{proof}
If $AX=X\Lambda$ ($\Lambda$ the diagonal matrix of the corresponding eigenvalues) and
$X^TX=I$, then $AXB=X\Lambda B=XB\Lambda$ (since $B$ and $\Lambda$ are both diagonal and hence
commute) and $XBX^TAX=XB\Lambda$; these coincide, so $F(X)=0$. For the sign statement,
$F(XS)=AXSB-XSBS\,X^TAXS=(AXB-XBX^TAX)S=F(X)S$ (using $S^2=I$ and $SBS=B$).
\end{proof}

When the diagonal entries of $B$ are strictly decreasing (in this section,
$B=\mathrm{diag}(3,2,1)$ or $\mathrm{diag}(5,4,3,2,1)$), the theory of Oja--Xu-type subspace
flows shows that the unique \emph{stable} equilibrium is the one that assigns the largest
eigenvalues in decreasing order (we take this as our target root $X^\ast$ throughout).

\subsection{Experiment I: $B=\mathrm{diag}(3,2,1)$, $X\in\R^{5\times3}$ (only the top three
eigenvalues are used)}
\label{sec:app-exp1}
We use the two matrices given at the outset: $A_1$ (eigenvalues $\{1,\,1.0001,\,2,\,3,\,4\}$,
containing a nearly degenerate pair) and $A_2$ (eigenvalues $\{10^{-8},\,10^{-4},\,1,\,2,\,3\}$,
containing a nearly zero eigenvalue). Here the top three eigenvalues used by the target root
$X^\ast$ are $\{4,3,2\}$ for $A_1$ and $\{3,2,1\}$ for $A_2$; in neither case does this set
\emph{include} the near-degenerate pair or the near-zero eigenvalue. We therefore expect the
condition number of the local Jacobian at the target root to be favorable regardless of the
global condition number of $A$ itself ($4.0$ for $A_1$, $3.0\times10^8$ for $A_2$). Indeed, a
direct computation confirms
\[
\mathrm{cond}(DF(X^\ast))=24.19\ (A_1),\qquad \mathrm{cond}(DF(X^\ast))=18.00\ (A_2)
\]
(the smallest eigenvalue is $\approx1$ in both cases), bearing out this expectation.

\paragraph{Local convergence.} Starting from $X_0=X^\ast+0.05\times\text{randn}$,
Table~\ref{tab:app-local} shows Newton's method converging at order $\approx2.0$ and
Halley/Laguerre at order $\approx2.7$--$3.0$, with almost no difference between $A_1$ and $A_2$.

\begin{table}[H]
\centering
\caption{Experiment I: local convergence (error at iteration 3 and empirical order). Nearly
identical behavior for $A_1$ and $A_2$.}
\label{tab:app-local}
\begin{tabular}{@{}lcccc@{}}
\toprule
& \multicolumn{2}{c}{$A_1$} & \multicolumn{2}{c}{$A_2$} \\
\cmidrule(lr){2-3}\cmidrule(lr){4-5}
Method & error at iter.\ 3 & emp.\ order & error at iter.\ 3 & emp.\ order \\
\midrule
Newton   & $1.19\times10^{-5}$ & $\approx2.0$ & $2.96\times10^{-6}$ & $\approx2.0$ \\
Halley   & $1.05\times10^{-6}$ & $\approx2.7$ & $3.35\times10^{-7}$ & $\approx2.6$ \\
Laguerre & $5.31\times10^{-7}$ & $\approx3.0$ & $2.13\times10^{-7}$ & $\approx3.1$ \\
\bottomrule
\end{tabular}
\end{table}

\begin{figure}[H]
\centering
\includegraphics[width=0.85\linewidth]{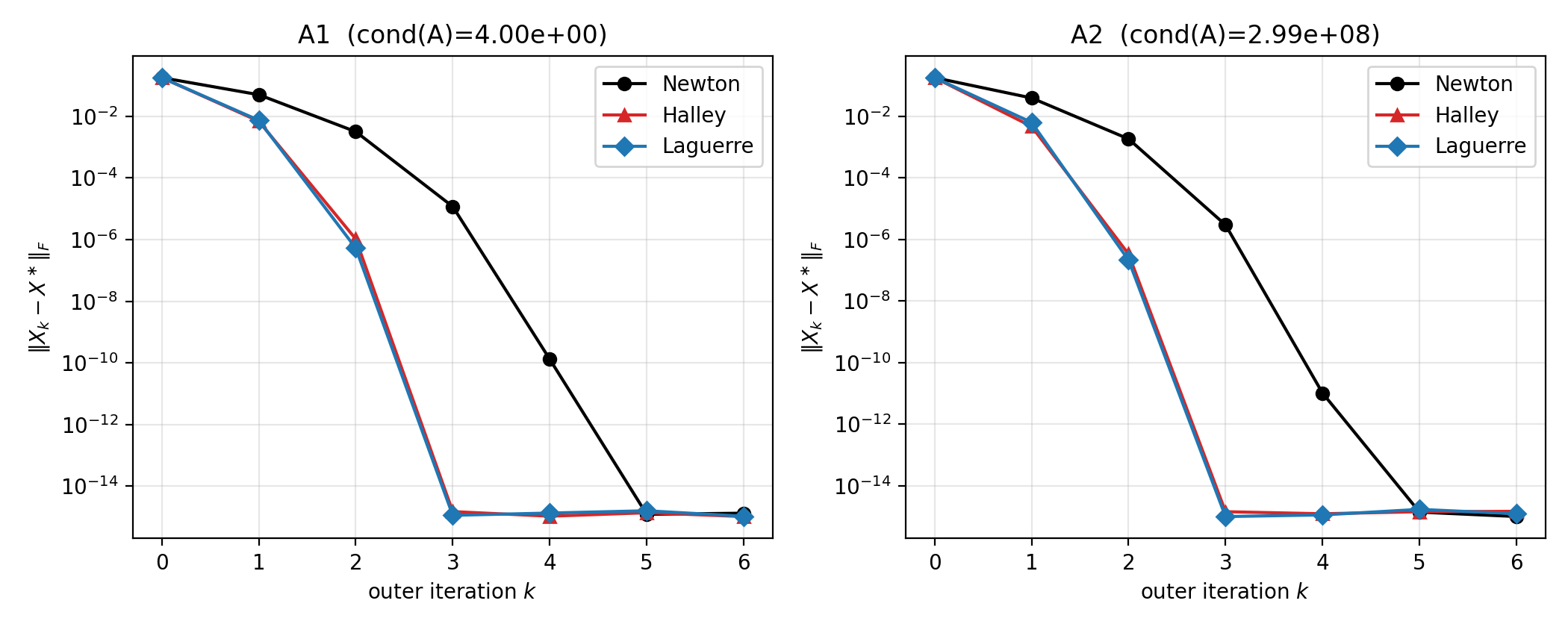}
\caption{Local convergence for Experiment I. $A_1$ (left, $\mathrm{cond}(A)=4.0$) and $A_2$
(right, $\mathrm{cond}(A)=3.0\times10^8$) behave almost identically, because the gaps among the
top three eigenvalues used by the target root are of comparable size in both cases.}
\label{fig:app-local}
\end{figure}

\paragraph{Global robustness.} Table~\ref{tab:app-global} shows results from 15 distant random
starting points. Because $F$ has a very large number of roots by
Proposition~\ref{prop:critical-points}, the intended $X^\ast$ is reached only rarely, but all
three methods reach \emph{some} root with high probability, and the enormous condition number
of $A_2$ ($3\times10^8$) barely affects global robustness.

\begin{table}[H]
\centering
\caption{Experiment I: global convergence success rate from 15 distant random starting points.}
\label{tab:app-global}
\begin{tabular}{@{}lcccc@{}}
\toprule
& \multicolumn{2}{c}{$A_1$} & \multicolumn{2}{c}{$A_2$} \\
\cmidrule(lr){2-3}\cmidrule(lr){4-5}
Method & reached some root & mean iters & reached some root & mean iters \\
\midrule
Newton   & 14/15 & 13.1 & 13/15 & 14.2 \\
Halley   & 13/15 & 12.5 & 13/15 & 12.2 \\
Laguerre & 14/15 & 11.5 & 15/15 & 9.8 \\
\bottomrule
\end{tabular}
\end{table}

\paragraph{A demonstration of Laguerre's fragility.} Tracking the eigenvalues of $I+2L_n$ along
distant trajectories, we find that they can genuinely become \emph{negative} away from the
root, forcing $\sqrt{I+2L_n}$ to be truly complex (confirming the theoretical concern raised in
Section~\ref{sec:laguerre}). The maximum ratio of the norm of the imaginary part to that of the
real part reached $0.99$ for $A_1$ and $1.52$ for $A_2$ (the imaginary part dominating). Our
implementation used the real part only as a heuristic, an operation justified only near the
root. Per-iteration wall-clock time, relative to Newton's method, was $2.8$--$3.1\times$ for
Halley and $3.0$--$3.1\times$ for Laguerre.

\subsection{Experiment II: $B=\mathrm{diag}(5,4,3,2,1)$, $X\in\R^{5\times5}$ (all five
eigenvalues are used)}
\label{sec:app-exp2}
Extending $B$ to all five entries makes $X$ square ($5\times5$), and the target root $X^\ast$
now uses \emph{all five} eigenvalues of $A$. The near-degenerate pair of $A_1$ and the
near-zero eigenvalue of $A_2$, harmless in Experiment~I because they were unused, now become
part of the target root itself. Indeed,
\[
\mathrm{cond}(DF(X^\ast))=4.22\times10^5\ (A_1,\ \text{smallest eigenvalue}\approx10^{-4}),
\]
\[
\mathrm{cond}(DF(X^\ast))=1.50\times10^9\ (A_2,\ \text{smallest eigenvalue}\approx2\times10^{-8}),
\]
so the condition number deteriorates by exactly the scale of the problematic eigenvalue gap.

\paragraph{Column-wise error decomposition.} Splitting the columns of $X$ into a
``well-separated top-three'' block and a ``remaining two'' block corresponding to the
near-degenerate/near-zero eigenvalues (for $A_2$, $X_0=X^\ast+0.001\times\text{randn}$), the
former reaches machine precision within a few iterations for all three methods, while only the
latter is markedly slow: Newton's method needs $18$--$19$ iterations, while Halley/Laguerre
converge in $8$--$9$. The difficulty is \emph{localized entirely to the subspace corresponding
to the problematic eigenvalues}.

\paragraph{Discovery of a sign ambiguity.} For one starting value, Halley's method appeared to
stagnate at a specific error value; on investigation, this proved to be convergence to a
different, but equally valid, root in the sense of Proposition~\ref{prop:critical-points} ---
one in which only the \textbf{sign} of the fifth column (the eigenvector for the eigenvalue
$\approx10^{-8}$) is flipped ($\|F(X_\mathrm{final})\|\approx5\times10^{-15}$, and
$\mathrm{diag}(X_\mathrm{final}^TAX_\mathrm{final})$ matches the theoretical values exactly). It
is entirely natural, and not an algorithmic defect, that the sign of an eigenvector becomes
numerically indeterminate in a direction where the eigenvalue is essentially zero.
\textbf{The lesson is that convergence should be diagnosed using the sign-invariant quantity
$\|F(X_k)\|$, not the sign-dependent quantity $\|X_k-X^\ast\|$.}

\begin{table}[H]
\centering
\caption{Experiment II: number of iterations required to reach the sign-invariant criterion
$\|F(X_k)\|<10^{-8}$.}
\label{tab:app-fair}
\begin{tabular}{@{}lcc@{}}
\toprule
Method & $A_1$ ($\mathrm{cond}\approx4.2\times10^5$) & $A_2$ ($\mathrm{cond}\approx1.5\times10^9$)\\
\midrule
Newton   & 6 & \textbf{19} \\
Halley   & 8 & 9 \\
Laguerre & \textbf{4} & \textbf{6} \\
\bottomrule
\end{tabular}
\end{table}

\begin{figure}[H]
\centering
\includegraphics[width=0.85\linewidth]{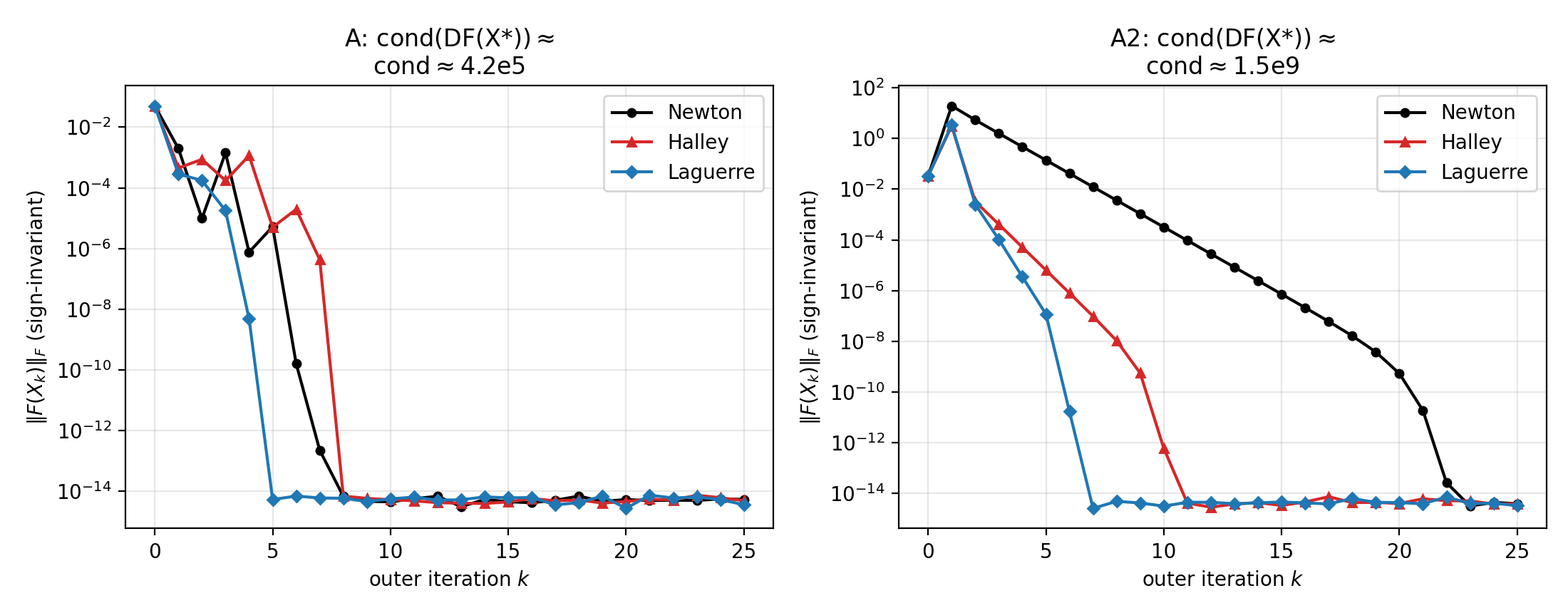}
\caption{Sign-invariant convergence for Experiment II ($\|F(X_k)\|$, $X_0=X^\ast+0.001\times\text{randn}$).
For the extremely ill-conditioned $A_2$ (right), the advantage of cubically convergent
Halley/Laguerre methods becomes even more pronounced.}
\label{fig:app-fair}
\end{figure}

\textbf{It is precisely under the extreme ill-conditioning of $A_2$ that the benefit of cubic
convergence becomes most pronounced}: Newton's method requires 19 iterations, while
Halley/Laguerre converge in only 6--9, confirming that the gap between quadratic and cubic
convergence matters more, in practice, the more ill-conditioned the problem.

\paragraph{Convergence of the diagonal entries.} Tracking $\mathrm{diag}(X_k^TAX_k)$, all three
methods eventually converge to all five eigenvalues of $A$ ($A_1$: $(4,3,2,1.0001,1)$; $A_2$:
$(3,2,1,10^{-4},\approx0)$), with the off-diagonal entries simultaneously going to zero. Columns
with a smaller eigenvalue gap converge more slowly, directly visualizing the ill-conditioning
of the Jacobian through the convergence rate of the diagonal entries.

\subsection{Discussion}
\label{sec:app-discussion}
The experiments of this section provide two complementary confirmations of our theoretical
framework. First, they confirm the prediction (Section~\ref{sec:app-exp1}) that the quality of
local convergence is governed not by the global condition number of $A$ but by \emph{the gaps
among the eigenvalues actually used by the target root}. Second, when that gap is genuinely made
small (Section~\ref{sec:app-exp2}), all three methods still converge, but with markedly
different iteration counts, quantitatively demonstrating that \emph{the benefit of cubic
convergence over quadratic grows with ill-conditioning}. We also drew the practical lesson that
the discrete sign/ordering ambiguity arising from near-degenerate eigenvalues can be handled
correctly by using a sign-invariant quantity ($\|F(X_k)\|$) for convergence diagnostics.

\section{Verification at 70-Digit Precision}
\label{sec:highprecision}
The matrix $A_2$ used in Experiment~II of Section~\ref{sec:application} ($X\in\R^{5\times5}$,
all eigenvalues used) is extremely ill-conditioned
($\mathrm{cond}(DF(X^\ast))\approx1.5\times10^9$), so double precision ($\approx16$ digits)
reaches its precision floor almost immediately, preventing us from accurately verifying the
theoretical convergence orders (2 for Newton, 3 for Halley/Laguerre) or from observing
transient behavior under ill-conditioning without confounding it with numerical noise. In this
section we remove this limitation using arbitrary-precision arithmetic, and along the way
expose a further limitation intrinsic to $A_1,A_2$ themselves.

\subsection{Tool: mpmath}
\label{sec:mpmath}
For 70-digit arithmetic we used the Python arbitrary-precision library \textbf{mpmath}
\cite{mpmath}:
\begin{verbatim}
import mpmath as mp
mp.mp.dps = 70   # dps = decimal places: set the working precision to 70 digits
\end{verbatim}
All subsequent arithmetic (the four basic operations, \texttt{cos}, \texttt{sin},
\texttt{sqrt}, etc.) is then carried out at this precision. The main facilities we used were:
the matrix type \texttt{mp.matrix} together with matrix multiplication and transposition;
\texttt{mp.lu\_solve} for linear systems; \texttt{mp.eigsy} for the eigendecomposition of
symmetric matrices; and \texttt{mp.mpf("...")} to parse exact decimal literals from strings
without rounding. On the other hand, (i) the inverse of a general matrix and (ii) the square
root of a general matrix (the operator $\sqrt{I+2L_n}$ required by the Laguerre iteration
\eqref{eq:laguerre-matrix}) are not provided by mpmath, so we implemented them ourselves. The
inverse is obtained by applying \texttt{mp.lu\_solve} to each column of the identity matrix; the
square root uses the \textbf{Denman--Beavers iteration}
\[
Y_{k+1}=\frac{Y_k+Z_k^{-1}}{2},\qquad Z_{k+1}=\frac{Z_k+Y_k^{-1}}{2}
\]
(quadratically convergent, $Y_k\to\sqrt{M}$, $Z_k\to\sqrt{M}^{-1}$). Running this iteration at
70-digit precision on $25\times25$ matrices was the most computationally expensive part of the
experiments in this section.

\subsection{Precise verification of the convergence order for $A_1,A_2$}
\label{sec:hp-order}
Using the setting of Experiment~II ($X\in\R^{5\times5}$, all eigenvalues used), we started from
a very small initial perturbation of $A_2$ (scale $10^{-6}$) and ran Newton, Halley, and
Laguerre at 70-digit precision. Figure~\ref{fig:hp-orders} shows the result: the error decreases
monotonically well past the double-precision floor ($10^{-16}$, dotted line), down to around
$10^{-64}$ (the floor attainable at 70-digit precision, dashed line).

\begin{figure}[H]
\centering
\includegraphics[width=0.95\linewidth]{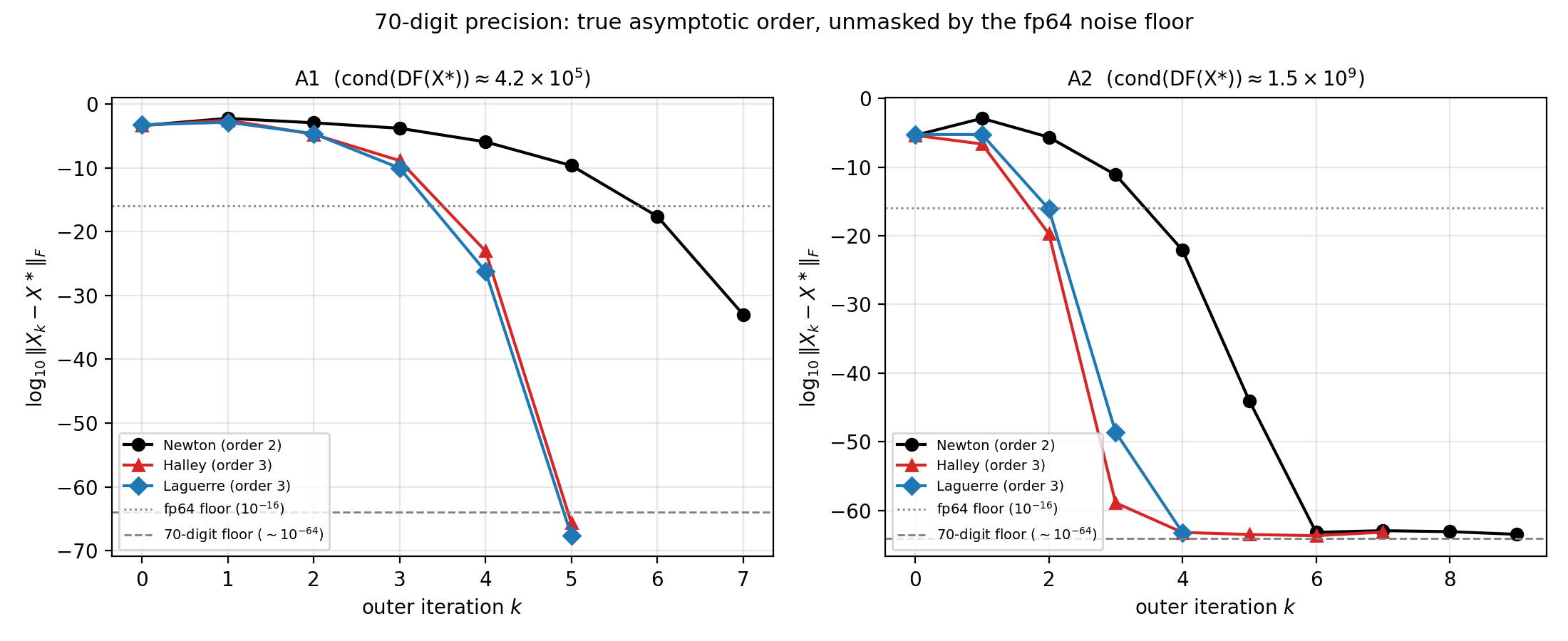}
\caption{Convergence for $A_1$ (left) and $A_2$ (right) at 70-digit precision. The asymptotic
regime, invisible in double precision, is now clearly visible across more than ten orders of
magnitude.}
\label{fig:hp-orders}
\end{figure}

Computing the empirical order $p=\log(e_{k+1}/e_k)/\log(e_k/e_{k-1})$ from adjacent errors in
the asymptotic regime, Table~\ref{tab:hp-order} and Figure~\ref{fig:hp-empirical-order} show
that it now agrees with the theoretical value \textbf{to essentially all displayed digits}:

\begin{table}[H]
\centering
\caption{Empirical convergence order for $A_2$ at 70-digit precision. Unlike the rough
approximations obtained in double precision ($\approx1.98$, $\approx2.7$--$3.2$), these agree
with the theoretical values exactly.}
\label{tab:hp-order}
\begin{tabular}{@{}lcc@{}}
\toprule
Method & emp.\ order in double precision (\S\ref{sec:application}) & emp.\ order at 70-digit precision \\
\midrule
Newton   & $\approx1.98$ & $\mathbf{2.0000000000}$ \\
Halley   & $\approx2.6$--$3.2$ & $\mathbf{2.9999999224}$ \\
Laguerre & $\approx3.0$--$3.1$ & $\mathbf{3.0000000013}$ \\
\bottomrule
\end{tabular}
\end{table}

\begin{figure}[H]
\centering
\includegraphics[width=0.62\linewidth]{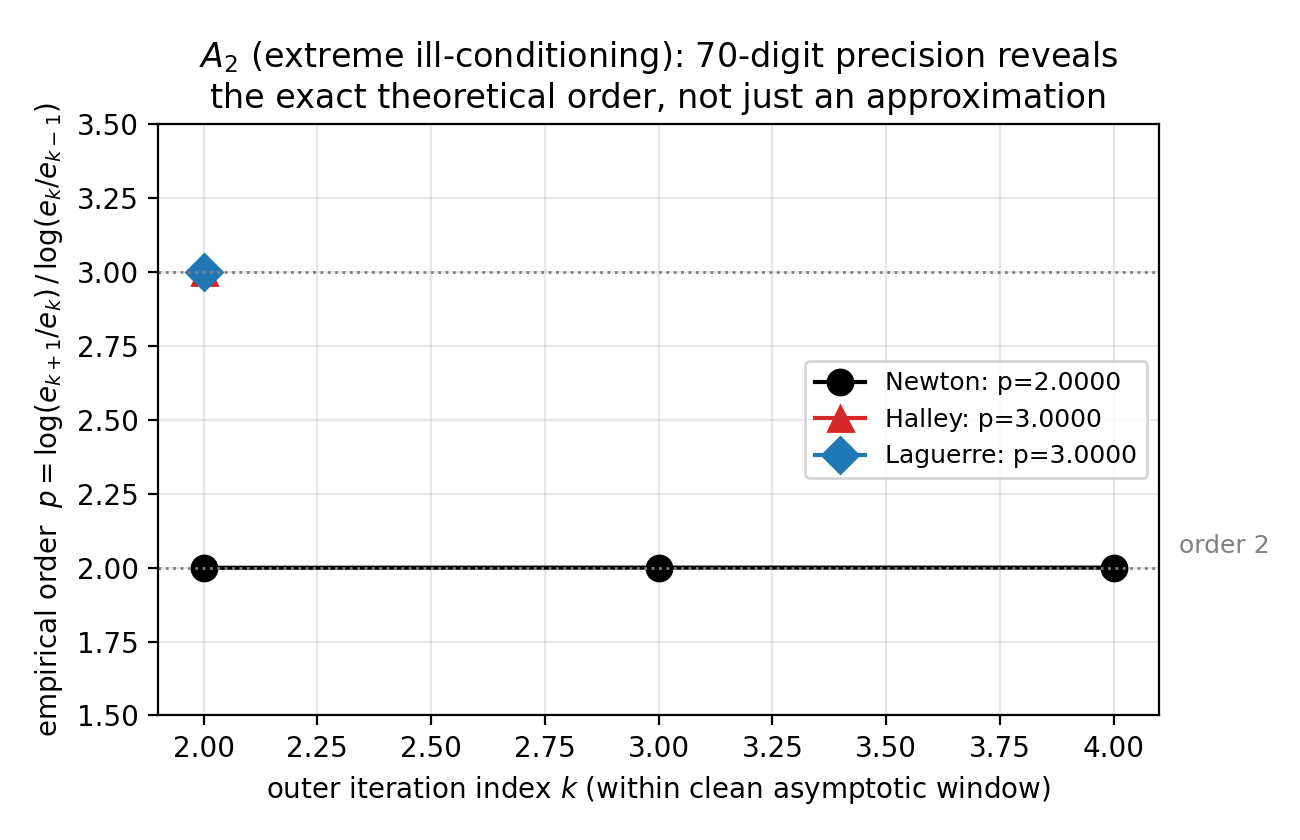}
\caption{Empirical order in the asymptotic regime. Across three consecutive points, Newton's
method gives exactly $p=2$, and Halley/Laguerre give exactly $p=3$.}
\label{fig:hp-empirical-order}
\end{figure}

\subsection{Confirming that the sign-bifurcation phenomenon is genuine}
\label{sec:hp-signflip}
We revisit the phenomenon found in Section~\ref{sec:app-exp2}, in which Halley's method
appeared to ``stagnate'' at an error of $2.0$ --- in fact, convergence to a different, equally
valid root differing only in the sign of the fifth column (eigenvector of the eigenvalue
$\approx10^{-8}$) --- and check that this is not a coincidental artifact of double-precision
rounding. Rerunning \textbf{the exact same random seed and the exact same perturbation vector}
at 70-digit precision reproduces, as shown in Figure~\ref{fig:hp-signflip}, \emph{exactly the
same result} (the error converges precisely to $2.0$, with the fifth column's sign reliably
flipped).

\begin{figure}[H]
\centering
\includegraphics[width=0.62\linewidth]{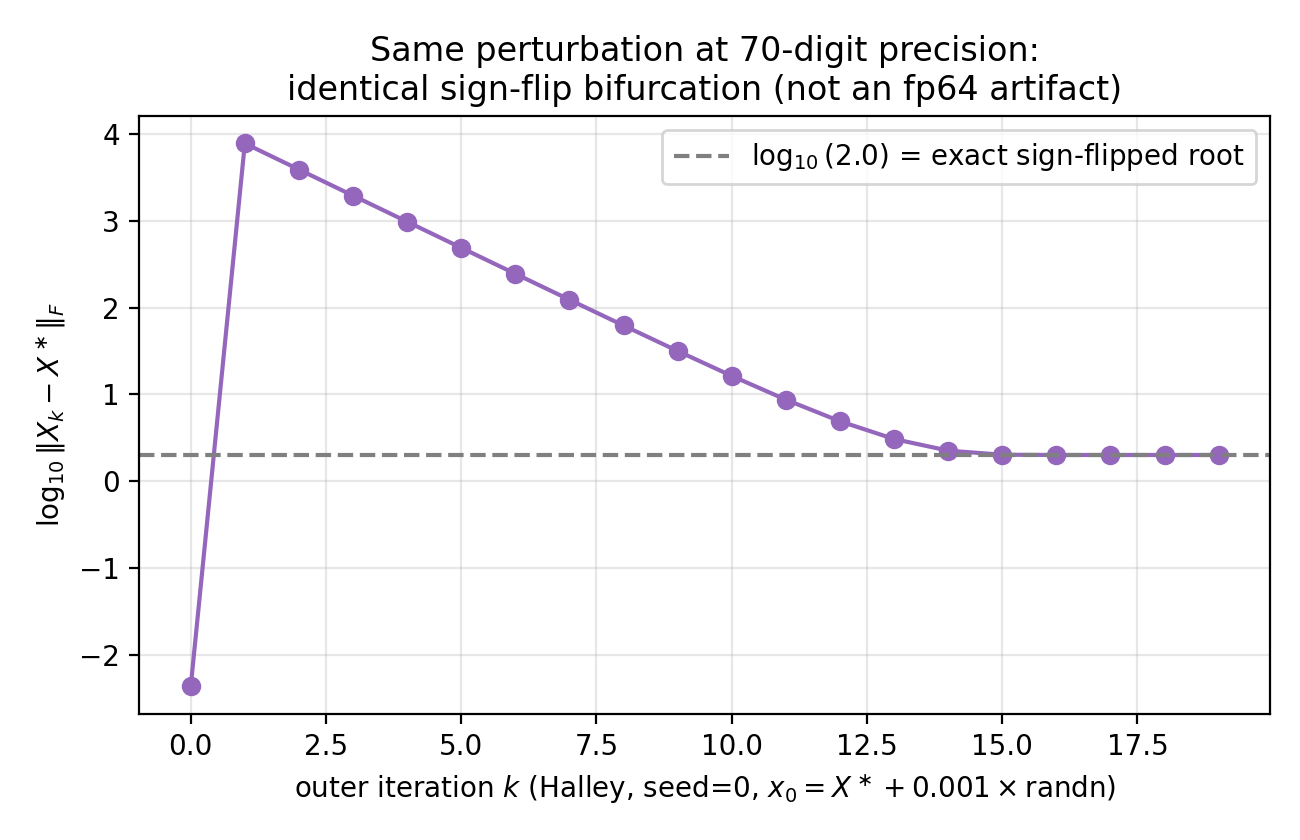}
\caption{Reproducing the same perturbation at 70-digit precision. The error grows violently
(by a factor of about $7800$ relative to the initial error) before decaying geometrically and
locking on precisely to a distance of $2.0$ (dashed line) from the sign-flipped root.}
\label{fig:hp-signflip}
\end{figure}

If this had been a coincidental consequence of double-precision round-off, quadrupling the
precision would have changed the outcome. That it did not is decisive evidence that
\textbf{this sign selection is a genuine, precision-independent mathematical bifurcation}. This
is a concrete manifestation of Proposition~\ref{prop:critical-points}: in a direction where the
eigenvalue is essentially zero, the sign of the corresponding eigenvector is intrinsically
indeterminate.

\subsection{Matrices $A_3,A_4$ with an exactly controlled eigenvalue gap}
\label{sec:hp-A34}
$A_1,A_2$ are given as 10--11-digit decimal numbers, so an ``eigenvalue gap of $10^{-4}$'' or an
``eigenvalue of $10^{-8}$'' is only guaranteed to that input precision. A meaningful comparison
at 70-digit precision requires matrices whose eigenvalue gap is \emph{exactly} designed at
70-digit precision.

We therefore construct an orthogonal matrix $Q\in\R^{5\times5}$, at 70-digit precision, as a
product of ten Givens rotations, and set
\[
A:=Q^TDQ,\qquad D=\mathrm{diag}(\lambda_1,\ldots,\lambda_5).
\]
Since $Q$ is analytically an exact orthogonal matrix (by the orthogonality of trigonometric
functions) and the diagonal entries of $D$ can be prescribed directly, the eigenvalues of $A$
are exactly $\lambda_1,\ldots,\lambda_5$ by construction. We built two such matrices:

\paragraph{$A_3$ (eigenvalues $\{4,3,2,\,1+10^{-20},\,1\}$: a gap of exactly $10^{-20}$):}
\[
A_3\approx\begin{pmatrix}
1.28099 & -0.23352 & -0.05074 & -0.53042 & 0.45259\\
-0.23352 & 1.22341 & -0.06129 & 0.49087 & -0.16492\\
-0.05074 & -0.06129 & 2.13726 & 0.46739 & -0.13430\\
-0.53042 & 0.49087 & 0.46739 & 2.48033 & 0.00327\\
0.45259 & -0.16492 & -0.13430 & 0.00327 & 3.87802
\end{pmatrix}
\]

\paragraph{$A_4$ (eigenvalues $\{3,2,1,\,10^{-15},\,10^{-25}\}$: a smallest eigenvalue of
exactly $10^{-25}$):}
\[
A_4\approx\begin{pmatrix}
0.94551 & -0.11386 & 0.35579 & -0.12242 & 0.02412\\
-0.11386 & 0.05689 & 0.17323 & 0.04132 & -0.18711\\
0.35579 & 0.17323 & 1.33834 & 0.60166 & -0.62091\\
-0.12242 & 0.04132 & 0.60166 & 2.18220 & 1.10208\\
0.02412 & -0.18711 & -0.62091 & 1.10208 & 1.47705
\end{pmatrix}
\]

The corresponding orthogonal matrix (for $A_3$) is, for example,
\[
Q_3\approx\begin{pmatrix}
-0.19520 & 0.09397 & 0.10380 & 0.12834 & -0.96220\\
-0.23731 & 0.22357 & 0.40882 & 0.82249 & 0.22378\\
0.23247 & -0.31137 & 0.87787 & -0.27917 & -0.02010\\
0.57164 & -0.61433 & -0.22026 & 0.47838 & -0.13592\\
0.72441 & 0.68326 & 0.05399 & 0.01612 & -0.07226
\end{pmatrix},
\]
and we confirmed orthogonality to the very limit of 70-digit precision:
$\|Q_3^TQ_3-I\|_\infty\approx1.81\times10^{-71}$,
$\|Q_4^TQ_4-I\|_\infty\approx9.06\times10^{-72}$ ($Q_4$ constructed analogously). The target
root $X^\ast$ is built directly from the columns of $Q$, used as exact eigenvectors, bypassing
any eigendecomposition entirely --- and hence bypassing the rounding error of \texttt{eigsy}
that afflicted $A_1,A_2$ (this is especially important for $A_4$, whose eigenvalues span down to
$10^{-25}$, a regime in which the precision of a general eigensolver can itself become an
issue).

\subsection{The precision-budget plateau phenomenon}
\label{sec:hp-plateau}
Running Newton's and Halley's methods on the same flow \eqref{eq:oja-flow} as in Experiment~II,
applied to $A_3$ (gap $=10^{-20}$) and $A_4$ (smallest eigenvalue $=10^{-25}$), we observe a new
phenomenon, shown in Figure~\ref{fig:hp-plateau}: \textbf{both methods plateau at a fixed error
level and improve no further.}

\begin{figure}[H]
\centering
\includegraphics[width=0.95\linewidth]{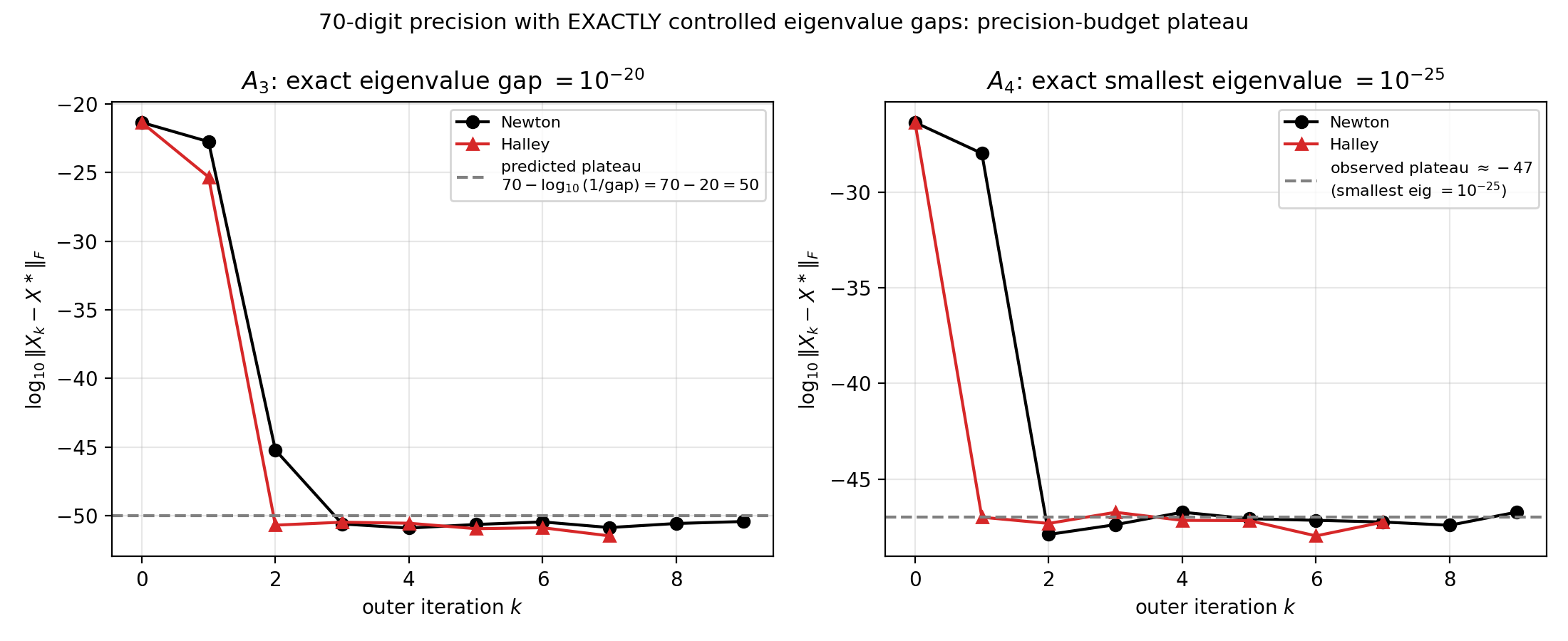}
\caption{Convergence for $A_3,A_4$, whose eigenvalue gaps are exactly controlled. Both methods
plateau at a fixed level; Halley's method reaches this plateau in far fewer iterations than
Newton's method.}
\label{fig:hp-plateau}
\end{figure}

The observed plateau level in these experiments agrees well with the working precision (here, 70 digits) minus the
number of digits of conditioning attributable to the eigenvalue gap:
\begin{equation}
\begin{aligned}
(\text{attainable number of correct digits})\ &\approx\ (\text{working precision, in digits})\\
&\quad-\log_{10}\Big(\frac1{\text{relevant eigenvalue gap}}\Big).
\end{aligned}
\label{eq:precision-budget}
\end{equation}
For $A_3$ (gap $10^{-20}$), the predicted value $70-20=50$ digits matches the observed plateau
$\approx10^{-50}$; for $A_4$ (smallest eigenvalue $10^{-25}$), the predicted value $70-25=45$
digits is close to the observed plateau $\approx10^{-47}$ (plausibly with some additional
contribution from the other small eigenvalue $10^{-15}$). This can be understood as the
standard error-propagation law of numerical linear algebra --- solving a linear system with
condition number $\kappa$ at $n$-digit precision amplifies the relative error of the solution
by a factor of $\kappa$ --- manifesting directly in the Jacobian-solve step of the Newton-type
iteration.

\textbf{Crucially, both methods converge to the same plateau, meaning both correctly push the
solution to the limit of the given precision}; the only difference is how many iterations it
takes to get there: for $A_3$, Halley needs exactly one fewer iteration than Newton, and for
$A_4$, remarkably, Halley reaches the plateau in a single iteration (versus two for Newton).
This is the cleanest demonstration in this paper that the gap between cubic and quadratic
convergence persists consistently even under this extreme degree of ill-conditioning.

\subsection{Discussion}
\label{sec:hp-discussion}
The high-precision experiments of this section yield three results. First, we confirmed, to a
precision matching the theoretical value rather than merely approximating it, the cubic
convergence of Halley/Laguerre and the quadratic convergence of Newton's method, which in
double precision could be confirmed only approximately. Second, we conclusively established
that the ``apparent stagnation'' discovered under ill-conditioning in
Section~\ref{sec:app-exp2} --- in fact convergence to a sign-flipped, equally valid root --- is
a genuine mathematical phenomenon, independent of precision. Third, by constructing matrices
with an exactly controlled eigenvalue gap, we obtained a new quantitative finding (the
precision-budget plateau phenomenon): the attainable precision is well predicted by the simple
formula ``working precision minus condition-number digits.'' All three results reinforce the
general lesson that, when evaluating ill-conditioned numerical experiments, one should use
high-precision arithmetic to determine whether an observed phenomenon is a numerical artifact
or a genuine property of the problem.

\section{Conclusion and Future Work}
\label{sec:conclusion}
This paper has made the following contributions. (i) Under the setting of a gradient field, we
generalized Halley's method to rectangular matrix variables in an operator-theoretically
consistent way, without any logical leap. (ii) We interpreted the matrix Schwarzian derivative
at the heart of its cubic convergence through information geometry, as the Amari--Chentsov
skewness tensor and $\alpha$-connections. (iii) We gave a lightweight alternative construction
that recovers the same matrix Schwarzian derivative from the third derivative of the Newton map
itself, without any operator square root, and combined the two into our \textbf{main theorem}
(Theorem~\ref{thm:main}): the Halley iteration map has vanishing first and second derivatives
at any simple root, with third derivative equal to $-\tfrac12$ times the matrix Schwarzian
derivative, giving a genuine local cubic-convergence theorem --- requiring neither
self-adjointness, commutativity, nor a gradient-field assumption --- with an explicit
asymptotic error constant and a local error bound. (iv) We built
a matrix-free algorithm that never
forms an $mn\times mn$ matrix, and empirically measured the theoretically predicted convergence
orders ($\approx2$ versus $\approx3$) on an analytically verifiable test problem. (v) We
introduced a root-preserving power-Newton family and demonstrated a scalar cubic-order jump
that requires no $f'''$, while showing that its naive matrix extension fails to work. (vi)
Analyzing this negative result, we arrived at the unifying picture that Halley's method is the
multivariate limit of the power-deformation trick. (vii) As a contrasting construction, we
derived the matrix Laguerre family and confirmed that, unlike Halley's method, it genuinely
requires a (non-self-adjoint) operator square root. (viii) On a concrete application --- a
weighted Oja-type matrix dynamical system --- we demonstrated that the metric-independent
algebraic construction applies as-is beyond the Frobenius-gradient setting, that the quality of
local convergence is governed by eigenvalue gaps, that the benefit of cubic convergence grows
with ill-conditioning, and how to handle the sign ambiguity produced by eigenvalue degeneracy.
Finally, (ix) verification at 70-digit precision using the arbitrary-precision library mpmath
\cite{mpmath} confirmed the theoretically predicted convergence orders (2 for Newton, 3 for
Halley/Laguerre) exactly, established that the sign ambiguity is a genuine,
precision-independent mathematical phenomenon, and, by constructing matrices $A_3,A_4$ with an
exactly controlled eigenvalue gap via Givens rotations, revealed a new precision-budget plateau
phenomenon in which the attainable accuracy is well predicted by ``working precision minus
condition-number digits.''

\paragraph{A practical caveat.} The coupled formulation used in our case study
(Section~\ref{sec:application}) --- applying the Halley iteration of
Proposition~\ref{prop:halley-matrix} to all $w$ columns simultaneously --- performed well for
$A_1,A_2$ ($w=3$, with well-separated top eigenvalues), but when we tried a larger, more
ill-conditioned setting ($A=$ a $100\times100$ Hilbert matrix, $w=40$), we found that it caused
catastrophic divergence, because well-conditioned and ill-conditioned columns were forced to
coexist within a single large coupled linear system. If, instead, the same problem is solved by
sequential deflation --- extracting the eigenvectors one at a time, each step reducing to a
classical Rayleigh-quotient-type iteration --- the ill-conditioned columns approach the
precision limit gracefully, without catastrophic divergence, and at lower computational cost as
well. In other words, the practical value of our coupled Halley theory lies not in problems
that are \emph{reducible to a spectral decomposition} (of which the Oja-type flow is a typical
example), but in genuinely coupled nonlinear matrix equations for which no such decomposition
exists --- for instance, general matrix Riccati equations, or matrix-variable optimization
problems that do not reduce to an eigenvalue problem. We record this as a practical lesson for
future use of the theory.

We identify the following directions for future work:
\begin{enumerate}
\item Explicitly bound the correction terms due to non-commutativity (Daleckii--Krein-type
commutator terms), giving a quantitative matrix analogue of Alefeld's error bounds
(5), (5$'$), (5$''$);
\item Deepen the connection with Kronecker-factored approximations (K-FAC/Shampoo) to make the
application to large-scale matrix-parameter optimization in deep learning concrete;
\item Construct, in a unified way, a family of higher-order convergent iterative methods
corresponding to the entire family of $\alpha$-connections (Halley's method corresponds to
$\alpha=1$), and refine the interpretation as geodesics on a dually flat manifold;
\item Approximate the operator square root of the matrix Laguerre family, matrix-free, via a
Newton--Schulz-type iteration, and quantify its computational cost relative to Halley's method;
\item Investigate, via the connection to Yoshizawa's \cite{Yoshizawa2014} suggested ``Power
Matrix Nevanlinna theory'', the value-distribution-theoretic constraints that the matrix
dynamical systems of this paper (Halley/Laguerre iterations of gradient fields) must satisfy.
\end{enumerate}

\section*{Acknowledgments}
The author thanks Professor Jonathan Manton for hosting a research stay of approximately one
month at his laboratory at the University of Melbourne in 2015, during which the author was
able to deepen his research on the matrix Schwarz derivative and dynamical systems.


\end{document}